\documentclass[12pt]{article}

\usepackage{amsmath}
\usepackage{amssymb}
\usepackage{amsfonts}
\usepackage{fullpage}
\usepackage{graphicx}
\usepackage{amsthm}
\usepackage{enumerate}
\usepackage{enumitem}
\usepackage{color}
\usepackage{cite}
\usepackage{enumitem}
\usepackage{color}
\usepackage{cite}
\usepackage{mathrsfs}
\usepackage[final]{showlabels}
\usepackage [english]{babel}
\usepackage [autostyle, english = american]{csquotes}
\MakeOuterQuote{"}

\setlist[enumerate,1]{label=(\arabic*)}
\setlist[enumerate,3]{label=(\roman*)}

\renewcommand{\phi}{\varphi}

\newtheorem{thm}{Theorem}[section]
\newtheorem{lem}[thm]{Lemma}
\newtheorem{cor}[thm]{Corollary}

\theoremstyle{definition}
\newtheorem{defn}[thm]{Definition}

\newtheorem*{note}{Note}
\newtheorem*{nota}{Notation}

\newcommand{\restricted}{\upharpoonright}
\newcommand{\ang}[1]{\langle#1\rangle}
\renewcommand{\phi}{\varphi}

\newcommand{\Lat}{\mathscr{L}}

\newcommand{\forces}{\Vdash}

\renewcommand{\a}{\textbf{a}}
\renewcommand{\b}{\textbf{b}}
\renewcommand{\c}{\textbf{c}}
\renewcommand{\d}{\textbf{d}}
\newcommand{\e}{\textbf{e}}

\newcommand{\Deg}{\mathscr{D}}

\newcommand{\M}{\mathcal{M}}
\renewcommand{\L}{\mathcal{L}}
\renewcommand{\P}{\mathcal{P}}
\newcommand{\Q}{\mathcal{Q}}
\newcommand{\condition}[1]{(\Phi_{#1},\X_{#1})}
\newcommand{\btau}{\boldsymbol{\tau}}
\newcommand{\X}{\textbf{X}}

\newcommand{\0}{\underline{0}}
\newcommand{\1}{\underline{1}}
\newcommand{\2}{\underline{2}}

\newcommand{\N}{\mathcal{N}}
\newcommand{\D}{\mathcal{D}}

\newcommand{\generic}{\Phi_{\texttt{G}}}

\newcommand{\infinitaryL}{\L^r_{\omega_1,\omega}(C)}
\newcommand{\infinitaryLB}{\L^r_{\omega_1,\omega}(B\oplus C)}

\newcommand{\T}{T(\Phi_0,\psi,k)}
\newcommand{\U}{U(\Phi_0,\ang{e,\alpha},k)}

\newcommand{\TPC}{T_{\P,C}(\Phi_0,\psi,k)}
\newcommand{\UPC}{U_{\P,C}(\Phi_0,\ang{e,\alpha},k)}
\newcommand{\TQC}{T_{\Q,B\oplus C}(\Phi_0,\psi,k)}
\newcommand{\UQC}{U_{\Q,B\oplus C}(\Phi_0,\ang{e,\alpha},k)}

\renewcommand{\u}{\vec{u}}
\newcommand{\m}{\vec{m}}
\newcommand{\umtf}{use-monotone Turing functional }
\newcommand{\umtfs}{use-monotone Turing functionals }

\title{On the decidability of the $\Sigma_2$ theories of the arithmetic and hyperarithmetic degrees as uppersemilattices}
\author{James Barnes}
\date{}

\begin{document}

\maketitle

\begin{abstract}
	We establish the decidability of the $\Sigma_2$ theory of both the arithmetic and hyperarithmetic degrees in the language of uppersemilattices i.e. the language with $\leq, 0$ and $\sqcup$. This is achieved by using Kumabe-Slaman forcing - along with other known results - to show that given finite uppersemilattices $\M$ and $\N$, where $\M$ is a subuppersemilattice of $\N$, then for both degree structures, every embedding of $\M$ into the structure extends to one of $\N$ iff $\N$ is an end-extension of $\M$.
\end{abstract}

\section*{Introduction}

	Given a transitive notion $r$ of reducibility between subsets of $\omega$ (we write $X \leq_r Y$ to mean $X$ is $r$-reducible to $Y$), the $r$ degrees (denoted $\Deg_r$) is the quotient of $2^\omega$ under the equivalence relation $\equiv_r$ of mutual reducibility: $X \equiv_r Y$ iff $X \leq_r Y$ and $Y \leq_r X$. The relation $\leq_r$ induces a partial order on $\Deg_r$, and given $X\subseteq \omega$, we call its equivalence class in $\Deg_r$ its $r$ degree. The primary such notion is Turing reducibility: $X\leq_T Y$ if there is a Turing machine with oracle $Y$ that computes the characteristic function of $X$. 

	Two related notions are arithmetic and hyperarithmetic reducibility (denoted $\leq_a$ and $\leq_h$ respectively). We say $X$ is arithmetic in $Y$ if $X$ is Turing reducible to the $n$th jump of $Y$ for some $n\in\omega$ (the jump $Y'$ of $Y$ is the set $\{e : \{e\}^Y(e)\downarrow\}$, the $n$th jump is the $n$th iterate of the jump operator). Equivalently, $X$ is arithmetic in $Y$ if is definable in first-order arithmetic with a predicate for membership in $Y$. 	$X$ is hyperarithmetic in $Y$ if $X$ is Turing reducible to the $\alpha$th jump of $Y$ for some $Y$-recursive ordinal $\alpha$ (at limit levels one takes effective unions). Equivalently, $X$ is hyperarithmetic in $Y$ if $X$ is $\Delta_1^1$ in $Y$ in second-order arithmetic. 

	The three degree structures $\Deg_T,\Deg_a,\Deg_h$ all have a least element: the degree of the empty set. Furthermore, each has a natural join operator induced by the operation $\oplus$ on $2^\omega$, which is defined by	
		\begin{align*}
			X\oplus Y = \{2n : n\in X\}\cup \{2n+1 : n\in Y\}.
		\end{align*}
	It is easy to check that the degree of $X\oplus Y$ depends only on the degree of $X$ and $Y$, and that it is their least upperbound. The reducibility, the $0$ degree, and the join operator endow each of these classes of degrees with the structure of an uppersemilattice with $0$. (We abbreviate this as USL).

\section{Preliminaries}

	A major project in recursion theory has been to classify the dividing line between the decidable and undecidable fragments of the theories of degree structures in various languages. In this paper we extend the result of Jockusch and Slaman\cite{JockuschSlaman1993} that the $\Sigma_2$ theory of $\Deg_T$ as a USL is decidable, to $\Deg_a$ and $\Deg_h$. This is sharp, in the sense that the $\Sigma_3$ theory of all three structures is undecidable, even with only $\leq$ in the language. This follows for any USL $\Lat$ into which every finite USL embeds as an initial segment. The result is originally due to Schmerl for $\Deg_T$, see Lerman - Thm VII.4.5\cite{Lerman1983} for a proof. For the initial segment results in $\Deg_a$ and $\Deg_h$, see Simpson\cite{Simpson1985}, and Kjos-Hanssen and Shore\cite{KjosHanssenShore2010} respectively.
	
	Our approach follows that of Jockusch and Slaman, who showed that given a USL $\Lat$, the $\Sigma_2$ theory of $\Lat$ is decidable if the following three facts obtain:
		\begin{enumerate}
			\item Every finite USL embeds as an initial segment of $\Lat$.
			\item For every $x\in \Lat$ and $n>0$ there are $y_0,\ldots,y_n\in \Lat$ which are mutually incomparable over $x$ (i.e. for each $i=0,\ldots, n$ the points $y_i$ and $x\sqcup \bigsqcup_{j\neq i} y_j$ are incomparable in $\Lat$).
			\item For any $a,b,c_i,d_i,e_i\in \Lat$,
					\[
					\left[\bigwedge_{i\in\omega} d_i \nleq c_i \, \&\, (a\nleq c_i\text{ or } d_i \nleq c_i \sqcup b)\right] \rightarrow 
					(\exists g)\left[b \leq a\sqcup g \, \& \bigwedge_{i\in\omega} d_i \nleq c_i \sqcup g \, \&  \bigwedge_{i\in\omega} g\nleq e_i\right]
					\label{goal}
					\]
		\end{enumerate}

	We call (3) the extended Posner-Robinson theorem for $\Lat$ (for the decidability result, one actually only needs the above to hold for arbitrarily long finite conjunctions, which is slightly weaker than allowing countable conjunctions).
	
	The analogues of (1) and (2) both hold in $\Deg_a$ and $\Deg_h$. Citations for (1) were given above, (2) follows by taking the columns of a sufficiently generic Cohen real (relativied to a suitable degree). (See Odifreddi\cite{Odifreddi1983} for an exposition of arithmetical Cohen forcing, and Sacks IV.3 \cite{Sacks1990} for hyperarithmetical Cohen forcing). Thus, the question of the decidability of the $\Sigma_2$ theory of $\Deg_a$ and $\Deg_h$ can be resolved by establishing the extended Posner-Robinson theorem for these structures. The rest of this paper is devoted to showing that this is indeed the case.
	
	Our method is a forcing construction, however, is different from that of Jockush and Slaman. They devise a forcing, depending on representatives $A\in \a$ and $B\in \b$, coding $B(n)$ into the top element of the $n$th column of their generic $g$. This renders $B \leq_T g'$, and so $B$ is both arithmetic and hyperarithmetic in $g$. If there are (arithmetic or hyperarithmetic) degrees $\c_i,\d_i$ such that $\d_i \leq \c_i \sqcup \b$ and $\a\nleq \c_i$, then $ d_i \leq \c_i \sqcup g$, and so such a forcing does not suffice in our settings.
		
	Instead, we use Kumabe-Slaman forcing, in particular, the presentation given in Shore and Slaman\cite{ShoreSlaman1999}. In the hyperarithmetic setting we also have the added concern of preserving $\omega_1^{\c_i}$ for each $\c_i$.
	
\section{The notion and language of forcing}

	\begin{defn}[Turing functionals]
		For $x\in \omega, y\in\{0,1\}$, and $\sigma\in 2^{<\omega}$, we call the tuple $\ang{x,y,\sigma}$ an \textbf{axiom} or a \textbf{computation}. A \textbf{Turing functional} is a set $\Phi$ of computations such that, if $\sigma_1$ and $\sigma_2$ are compatible, and $\ang{x,y_1,\sigma_1},\ang{x,y_2,\sigma_2}\in \Phi$, then $y_1 = y_2$ and $\sigma_1 = \sigma_2$.
		\\
		Furthermore, $\Phi$ is \textbf{use-monotone} if,
			\begin{itemize}
				\item whenever $\sigma_1 \subsetneq \sigma_2$ and $\ang{x_1,y_1,\sigma_1},\ang{x_2,y_2,\sigma_2}\in \Phi$, then $x_1<x_2$, and
				\item whenever $x_1<x_2$ and $\ang{x_2,y_2,\sigma_2} \in \Phi$, there is a $y_1$ and a $\sigma_1\subsetneq \sigma_2$ such that $\ang{x_1,y_1,\sigma_1}\in \Phi$.
			\end{itemize}
		\label{deftftnl}
	\end{defn}
			
	\begin{nota}
		We write $\Phi(x,\sigma)=y$ to mean there is a $\tau\subseteq \sigma$ such that $\ang{x,y,\tau}\in \Phi$. If $X\subseteq \omega$, we write $\Phi(x,X)=y$ to mean there is an $n$ such that $\Phi(x,X\restricted n) = y$, and write $\Phi(X)$ for the (partial) function evaluated in this way. Note that $\Phi(X)$ is (uniformly) recursive in the join of $X$ and $\Phi$.
	\end{nota}
	
	\begin{defn}[Kumabe-Slaman forcing]
		\textbf{Kumabe-Slaman forcing} is the partial order  $\mathcal{P}$, with conditions $p = (\Phi_p,\X_p)$, where $\Phi_p$ is a finite use-monotone Turing functional and $\X_p$ is a finite set of subsets of $\omega$ (hereafter we refer to such $X\subseteq \omega$ as \textbf{reals}). For elements $p,q$ of $\mathcal{P}$, we say $q \leq_\P p$ (\textbf{$q$ extends $p$}) if
			\begin{itemize}
				\item $\Phi_p \subseteq \Phi_q$ and for all $(x_q,y_q,\sigma_q)\in \Phi_q\setminus \Phi_p$ and all $\ang{x_p, y_p, \sigma_p} \in \Phi_p$ the length of $\sigma_q$ is greater than the length of $\sigma_p$.
				\item $\X_p \subseteq \X_q$.
				\item For every $x,y,$ and $X\in\X_p$, if $\Phi_q(X)(x)=y$, then $\Phi_p(X)(x)=y$ i.e. $q$ can't add new axioms $\ang{x,y,\sigma}$ where $\sigma\subset X$ for any $X\in \X_p$.
			\end{itemize}
		\label{ksforcing}
	\end{defn}
		
	We think of a forcing condition $p=\condition{p}$ as approximating a (possibly nonrecursive) Turing functional $\Phi_G$. To a descending sequence of conditions $(p_n)_{n=1}^\infty$ we associate the object $\Phi_G = \bigcup_n \Phi_{p_n}$, a use-monotone Turing functional.
			
	On a formal level, we code each finite use-monotone Turing functional $\Phi$ as a natural number, which encodes each element of $\Phi$ and also the size of $\Phi$. For instance, given some recursive encoding of triples $\ang{x,y,\sigma}$ as positive integers, we could code $\Phi = \{\ang{x_1,y_1,\sigma_1},\ldots,\ang{x_n,y_n,\sigma_n}\}$ as the product $\prod_{i=1}^n P(i)^{\ang{x_i,y_i,\sigma_i}}$ where $P(i)$ is the $i$th prime. Consequently, a forcing condition is coded as a pair, the first coordinate is a natural number and the second a finite set of reals. We dub the first coordinate the \textbf{finite part}, and the second the \textbf{infinite part}. Given finite \umtfs $\Phi_p$ and $\Phi_q$, we will occasionally abuse notation and write $\Phi_p\leq_\P \Phi_q$, by which we mean $(\Phi_p,\emptyset)\leq_\P (\Phi_q,\emptyset)$.
	
	We will often have occasion to modify our notion of forcing somewhat, either by insisting that certain axioms don't belong to the finite part, or the reals can only come from a particular subclass of $2^\omega$. For the first modification we introduce some notation:
	\begin{nota}
		Given reals $A$ and $B$ (think of representatives from $\a$ and $\b$ respectively), we say that a Turing functional $\Phi$ \textbf{partially computes $B$ on input $A$} if whenever $\Phi(A)(n)$ is defined, it equals $B(n)$. More explicitly, if $\ang{x,y,\sigma}\in \Phi$ is an axiom such that $\sigma \subseteq A$ (we say this axiom \textbf{applies to $A$}), then $y = B(x)$.	
	\end{nota} 
	
	\begin{defn}[Restricted Kumabe-Slaman forcing]
		$\Q_{A,B}$ is the subset of $\P$ containing all those conditions $q=\condition{q}$ such that $\Phi_q$ partially computes $B$ on input $A$, and $\X_q$ does not contain $A$. Extension in $\Q_{A,B}$ (denoted $\leq_{\Q_{A,B}}$) is defined as for $\P$. The partial order $\Q_{A,B}$ is a suborder of $\P$, and will be abbreviated as $\Q$ if it is understood what $A$ and $B$ are. We employ a similar convention as for $\P$ and write $\Phi_q \leq_\Q \Phi_r$ to mean $(\Phi_q,\emptyset)\leq_\Q (\Phi_r,\emptyset)$.
		\label{restksforcing}
	\end{defn}

	Our main goal is to construct a sufficiently generic sequence of conditions into which we code $B$ via $A$. More precisely, for each $n\in\omega$, our generic will have an axiom $(n,B(n),\alpha)$ where $\alpha$ is an initial segment of $A$. This will ensure that $\Phi_G(A)$ is the characteristic function of $B$, and so $B\leq_T \Phi_G \oplus A$. We will need to show that we can construct a generic sequence without adding "incorrect" axioms applying to $A$, and without adding $A$ into the infinite part of any condition (which would prevent us from adding any axioms about $A$ later).
	
	The forcing $\Q_{A,B}$ has all this hard-wired in, however, it is at the expense of checking whether $\Phi$ can serve as the finite part of the condition being $A\oplus B$ recursive, instead of plain old recursive. For $\P$, we need to show that we can manually avoid adding in bad axioms or $A$ into the infinite part of a condition, while still constructing a generic sequence.
	
	We ensure that $\d_i \nleq \c_i \sqcup g$ via genericity. At a high level, our method is quite standard, we will define a forcing language that will be powerful enough to express all the possible pertinent reductions from $C_i\oplus \Phi_G$ to $D_i$ (for chosen representatives $C_i\in \c_i$ and $D_i\in \d_i$).  We will try to diagonalize against each of these reductions in turn. The result for the arithmetic degrees is a direct corollary of our methods for the hyperarithmetic degrees, so we need not treat them separately.
	
	For forcing in the arithmetic or Turing degrees, a forcing language based upon first-order arithmetic often suffices. However, for the hyperarithmetic setting something more expressive is required. We use recursive infinitary formulas as presented in Ash and Knight chapter 9 \cite{AshKnight2000} (over the language of arithmetic, see Shore\cite{Shore2015} for an exposition). Roughly speaking, at the ground level we have the quantifier-free finitary sentences in our language, and then, for each recursive ordinal $\alpha$, a $\Sigma_\alpha^r$ formula is an r.e. disjunction of the form
		\[
			\bigvee_i (\exists \vec{u}_i) \psi_i(\vec{u}_i,\vec{x})
		\]
	where each $\psi_i$ is a $\Pi_{\beta_i}^r$ formula for some $\beta_i<\alpha$. Similarly a $\Pi_\alpha^r$ formula is an r.e. conjunction of universally quantified formulas all of which are $\Sigma_{\beta_i}^r$ for some $\beta_i < \alpha$.
	
	\begin{defn}[The language of forcing]
		We start with $\L$, the set containing the following nonlogical symbols:
			\begin{itemize}
				\item Numerals: $\0,\1,\2,\ldots$.
				\item A 5-ary relation for the universal recursive predicate: $\phi(e,x,s,y,\sigma)$ (to be interpreted as, the $e$th Turing machine on input $x$ running for $s$ steps with oracle $\sigma$ halts, and outputs $y$, we will often write this as $\{e\}_s^\sigma(x)\downarrow = y$, and write $\{e\}^X(x)$ for the partial function evaluated this way).
				\item A predicate for a string being an initial segment of our generic: $\sigma\subset\generic$
			\end{itemize}
		From $\L$ we build up the recursive infinitary formulas as indicated in Ash and Knight, with the convention that the ground level formulas are in CNF, and denote this collection $\L^r_{\omega_1,\omega}$.
		
		To relativise to some $S\subseteq \omega$, we add a predicate (written $\sigma\subset \texttt{S}$) for a string being an initial segment of $S$, and build up the $S$-recursive infinitary formulas i.e. we allow $S$-r.e. infinitary conjunctions and disjunctions. We denote the collection of such formulas $\L^r_{\omega_1,\omega}(S)$ and rank the formulas as $\Sigma_\alpha^r(S)$ and $\Pi_\alpha^r(S)$ for $\alpha < \omega_1^S$. 
		\label{deflang}
	\end{defn}
	
	We think of numerals as having a dual role in our language: they represent both numbers (in the obvious way) and also strings via some fixed recursive bijection between $2^{<\omega}$ and $\omega$. We will suppress the technical apparatus needed, and will use lowercase Roman letters when we are thinking in terms of numbers, and lowercase Greek letters when we are thinking in terms of strings.
	
	We can now define the forcing relation $\forces_\P$ for $\P$, with some $S\subseteq\omega$ as our set parameter.
	
	\begin{defn}[The forcing relation]
		The relation $\forces_\P$ between elements of $\P$ and sentences of $\L_{\omega_1,\omega}^r(S)$ is defined by induction on the ordinal rank and complexity of the formula.
		\begin{enumerate}
			\item $p \forces_\P \underline{n} = \underline{m}$ iff $n = m$.
			\item $p\forces_\P \phi(\underline{e},\underline{x},\underline{s},\underline{y},\underline{\sigma})$ iff $\{e\}_s^\sigma(x)\downarrow=y$.
			\item $p\forces_\P \underline{\sigma} \subset \texttt{S}$ iff $\sigma\subset S$.
			\item $p\forces_\P \underline{\sigma} \subset \generic$ iff for all $n$ less than $|\sigma|$:
				\begin{enumerate}
					\item If $\sigma(n) = 1$, then $n \in \Phi_p$.
					\item If $\sigma(n) = 0$, then either $n$ is not of the form $\ang{x,y,\tau}$, or $n=\ang{x,y,\tau}$, $n\notin \Phi_p$ and
						\begin{enumerate}
							\item there is an $\ang{x_0,y_0,\tau_0}\in \Phi_p$ such that $|\sigma_0|$ is greater than $|\sigma|$, or, $x_0$ is greater than $x$ and $\sigma_0$ is compatible with $\sigma$, or
							\item $\sigma$ is an initial segment of one of the elements of $\X_p$.
						\end{enumerate}
				\end{enumerate}
			\item For an atomic sentence $\psi$, $p\forces_\P \neg \psi$ iff there is no $q\leq_\P p$ with $q\forces \psi$.
			\item For sentences $\psi_1,\ldots,\psi_n$ that are literals, $p\forces_\P \psi_1 \lor \cdots \lor\psi_n$ iff for all $q\leq_\P p$ there exists an $r\leq_\P q$ and an $i\in\{1,\ldots,n\}$ such that $r \forces \psi_i$.
			\item For sentences $\psi_1,\ldots,\psi_n$ that are finite disjunctions of literals, $p\forces_\P \psi_1 \land\cdots \land \psi_n$ iff $p\forces \psi_i$ for each $i\in\{1,\ldots,n\}$.
			\item $p \forces_\P \bigvee\limits_i (\exists \vec{u}_i) \psi_i(\vec{u_i})$ iff there is some $i$ and some $\vec{n}$ such that $p\forces_\P \psi_i(\underline{\vec{n}})$.
			\item $p \forces_\P \bigwedge\limits_i (\forall \vec{u}_i) \psi_i(\vec{u_i})$ iff for each $i,\vec{n}$ and $q\leq_\P p$ there is some $r\leq_\P q$ such that $r\forces \psi_i(\underline{\vec{n}})$.
		\end{enumerate}
	\end{defn}
	
	\begin{note}
		We note some abuses of the language of forcing which we will employ hereafter. Firstly, we will stop underlining numerals.  Secondly, although our language only has negations of atomic formulas, there is an $S$-recursive function neg that takes (an index for) a formula $\psi$ and returns (an index for) a formula that is logically equivalent to the negation of $\psi$. Also, neg preserves the ordinal rank of the formula i.e. if $\psi$ is $\Pi^r_\alpha(S)$, then neg$(\psi)$ is $\Sigma^r_\alpha(S)$, and visa versa. Consequently, we write $\neg\psi$ freely as if negation were a part of the language.
	\end{note}
	
	As expected, it can be shown that, for each $S$-recursive ordinal $\alpha$, the claim that the partial function $\{e\}^{(S\oplus \Phi_G)^{(\alpha)}}$ is defined on an input $x$ (and, optionally, has value $y$) can be expressed as a $\Sigma_{\alpha+1}^r$ sentence in our language, uniformly in $e,x$ and $y$.
	
	The definition of $\forces_\Q$ is analogous except in clause 4. In particular we must add a third possibility to (b), namely, that $\tau \subset A$ but $B(x)\neq y$.
	
	Standard lemmas include that extension preserves forcing, no condition forces a sentence and its (formal) negation, and given a sentence and a condition, we can find an extension of that condition deciding the sentence.
	
	The countability of $\L^r_{\omega_1,\omega}(S)$ implies that generic sequences exist. To a generic sequence $\{p_n\}$ we associate the generic object $\Phi_G= \bigcup \Phi_{p_n}$, which is the intended interpretation of $\generic$. With this, one proves the final standard lemma: truth equals forcing. Proofs of these standard lemmas can be found in chapter 10 of Ash and Knight.
	
	Now a technical lemma about the complexity of the forcing relation.
	
	\begin{lem}
		Given a finite \umtf $\Phi_0$, and a quantifier-free finitary sentence $\psi$ in $\L^r_{\omega_1,\omega}(S)$ (i.e. a $\Pi_0^r$ sentence), we can $S$-recursively decide whether there is an $\X$ such that $(\Phi_0,\X)\in \P$ forces $\psi$, uniformly in both the sentence and the functional. 
		
		Furthermore, we can also make this decision for $\Q_{A,B}$ (and $\forces_\Q$), however, the procedure is uniformly recursive in $A\oplus B \oplus S$.
		\label{quantifierfreeforcing}
	\end{lem}
	
	\begin{proof}
		We begin with the claim for $\P$. By convention, finitary sentences are in conjunctive normal form. There are four types of atomic sentences
			\begin{enumerate}[label=(\arabic*)]
				\item $n= m$
				\item $\varphi(e,x,s,y,\sigma)$
				\item $\sigma \subset \texttt{S}$
				\item $\sigma \subset \Phi_\texttt{G}$
			\end{enumerate}
		then there are the negations of these atoms, finite disjunctions of such literals, and finite conjunctions of such disjunctions. 

		For the first three types of atoms, a condition forces it iff it is true, and we can recursively in $S$ check each of these facts. Similarly if one of these atoms is false, then no condition forces it to be true, and so every condition forces the negation. 

		The fourth type of atom is $\sigma \subset \generic$. We claim that there is an $\X$ so that $(\Phi_0,\X)\forces \sigma\subset \generic$ iff for each $n<|\sigma|$, we have $\Phi_0(n) = \sigma(n)$ (which is uniformly recursive to check). To see this, note that a condition $p$ forces $\sigma\subset\generic$ iff $\sigma(n) = \Phi_p(n)$ and for each $n$ such that $\sigma(n) = 0$, if $n = \ang{x,y,\tau}$, then $\ang{x,y,\tau}$ is prevented from entering any extension of $p$ (i.e. there is either an axiom $(x_0,y_0,\tau_0)\in \Phi_p$ with $|\tau_0| > |\tau|$, or $x_0 \geq x$ and $\tau_0$ and $\tau$ are compatible, or $\tau$ is an initial segment of some $X\in\X_p$).

		Therefore, for $p$ to force $\sigma\subset \generic$, it is necessary that $\Phi_p$ agrees with $\sigma$ everywhere that $\sigma$ is defined. If $\Phi_0$ and $\sigma$ do agree, then for each $n<|\sigma|$ with $\sigma(n)=0$ and $n$ of the form $\ang{x,y,\tau}$, we can pick a real $X$ extending $\tau$ and let $\X$ be the collection of such reals. Then $(\Phi_0,\X)\forces \sigma\subset \generic$ which completes the claim. 

		Now we turn to literals of the form $\neg(\sigma \subset \generic)$. We claim that we can pick an $\X$ to guarantee that $(\Phi_0,\X)\forces_\P \neg(\sigma\subset \generic)$ iff there is an $n<|\sigma|$ such that $\sigma(n)\neq \Phi_0(n)$. From the definition of forcing, a condition $p$ forces $\neg(\sigma\subset \generic)$ iff every extension of $p$ fails to force $\sigma \subset \generic$. Suppose $\Phi_0$ and $\sigma$ agree wherever $\sigma$ is defined, then, by the above, we can pick $\X_0$ so that $(\Phi_0,\X_0)\forces \sigma\subset \generic$. Therefore, for each $\X$ we have that $(\Phi_0,\X\cup \X_0)\leq_\P (\Phi_0,\X)$ and $(\Phi_0,\X\cup\X_0)\forces \sigma\subset\generic$. Consequently, if $\Phi_0$ and $\sigma$ agree wherever $\sigma$ is defined, then, for each choice of $\X$ there is an extension of $(\Phi_0,\X)$ forcing $\sigma\subset \generic$. By taking the contrapositive we have established one direction of the claim.

		For the other direction, suppose there is an $n$ so that $\sigma(n) \neq \Phi_0(n)$. For such an $n$, either $\sigma(n) = 0$ and $\Phi_0(n) = 1$, or, $\sigma(n) = 1$ and $\Phi_0(n) = 0$. In the first case, no extension of $(\Phi_0,\emptyset)$ can force $\sigma\subset \generic$ as $n$ is in every stronger condition. Hence, $\X = \emptyset$ suffices as the set of reals. In the second case, if $n\neq \ang{x,y,\tau}$, then $n$ is not an element of any Turing functional and so again $\X = \emptyset$ suffices. So suppose $n=\ang{x,y,\tau}$ and let $X$ be any real extending $\tau$. As $\ang{x,y,\tau}\notin \Phi_0$, it is also not in any extension of $(\Phi_0,\{X\})$, so no extension of this condition can force $\sigma\subset\generic$, and so, $\{X\}$ is our set of reals, which completes the proof of the claim.

		Now, suppose we have a sentence $\psi = l_1 \lor \cdots \lor l_n$, which is a finite disjunction of literals. From the definition of forcing, it is clear that if a condition forces one of the disjuncts, it forces the whole disjunction. So, using induction, ask whether there is a literal $l_i$ and a set of reals $\X$ so that $(\Phi_0,\X)\forces l_i$. If there is, we are done, as such an $\X$ suffices. If not,  we claim no $\X$ works. To see this, firstly, note that we get to ignore any of the literals with atomic parts of type (1), (2) or (3), as $(\emptyset,\emptyset)$ decides these sentences as false. So now assume $\psi$ is a disjunction of literals built from atoms of type (4). For each positive literal $\sigma_i \subset \generic$, because we can't pick an $\X$ to force it, then, $\Phi_0$ and $\sigma_i$ disagree somewhere on the domain of $\sigma_i$. Hence, for each $i$, there is a finite collection $\X_i$ so that $(\Phi_0,\X_i)\forces \neg(\sigma_i\subset \generic)$. Similarly, for each negative literal, as we can't pick an $\X$ to force it, then, $\Phi_0$ and $\sigma_i$ agree on the domain of $\sigma_i$, and so, we can pick $\X_i$ so that $(\Phi_0,\X_i)\forces \sigma_i\subset\generic$.
		
		So fix a set of reals $\X$, and define $\X_0 = \bigcup_{i=1}^n \X_i$. Then $(\Phi_0,\X_0\cup \X) \leq_\P (\Phi_0,\X)$ yet $(\Phi_0,\X_0\cup \X)$ forces the opposite of each literal in $\psi$, i.e. forces $\neg\psi$. Consequently, $(\Phi_0,\X)$ does not force $\psi$ for any $\X$.
		
		For conjunctions $\psi_1 \land \cdots \land \psi_m$, the definition of forcing is you force the conjunction iff you force each of the conjuncts. Inductively, ask if there are reals $\X_i$ such that $(\Phi_0,\X_i)\forces \psi_i$, if "yes" in each case, then, $(\Phi_0,\bigcup_{i=1}^m \X_i)$ forces each conjunct. If any of them return a "no", then no choice of reals can suffice for all conjuncts. This completes the decision procedure, uniformity is clear.
		
		For $\Q$, only minor changes are needed. Firstly, one must check whether $\Phi_0$ correctly computes $B$ on input $A$, which is recursive in $A\oplus B$. Secondly, although the definition of $\forces_\Q$ for formulas of the form $\sigma\subset \generic$ is slightly different i.e. an axiom $\ang{x,y,\sigma}$ where $\sigma\subset A$, yet $y\neq B(x)$ is not a member of any $\Q$ condition, this is, again $A\oplus B$ recursive to check. Finally, whenever we picked a real $X$ to prevent an axiom from entering any stronger condition, $X$ just had to extend some finite string, and so, need not be $A$. These are the only observations needed.
	\end{proof}
	
	Now we break the proof of the extended Posner-Robinson theorem into cases. For the moment, imagine we are trying to do a single instance of our goal:
		\[
			[\d\nleq_h \c \,\&\, (\a\nleq_h \c \text{ or } \d \nleq_h \b \sqcup \c )] \rightarrow (\exists g)[\b \leq_h \a \sqcup g \, \& \, \d\nleq_h \c\sqcup g]
		\]
	for some hyperarithmetic degrees $\a,\b,\c$ and $\d$. Suppose they satisfy the antecedent, then, either $\a\nleq_h \c$, or, $\a\leq_h \c$ and $\d\nleq_h \b \sqcup \c$. We call the former the easy case, and the latter the hard case. Pick representatives $A,B,C$ and $D$ from $\a,\b,\c$ and $\d$ respectively, furthermore, if we are in the hard case, ensure that $A \leq_T C$.
	
\section{The easy case}
	We now proceed with the easy case, for which we use the forcing $\P$. While constructing our generic, we will need to code $B$ into the join of $A$ and $\Phi_G$. Our coding strategy is to add axioms $\ang{n,B(n),\alpha}$ to our generic, with $\alpha \subset A$. To ensure that this is possible, we need to show that we can construct a generic sequence without adding $A$ to the infinite part of any condition, and, such that any axiom $\ang{x,y,\sigma}$ applying to $A$ that we add to the finite part satisfies $B(x) = y$.
	
	Additionally, we need to show that a sufficiently generic sequence ensures that $D \nleq_h C \oplus \Phi_G$. Towards this goal we use the forcing language $\infinitaryL$. As our language can only express reductions from some $C$-recursive jump of $C\oplus \Phi_G$ we must also show that $\Phi_G$ preserves $\omega_1^C$. Thus, we have three goals: coding $B$, diagonalizing against computing $D$, and preserving $\omega_1^C$. In this section we work towards the first two goals, leaving the last for Section 5.
	
	\subsection*{Coding $B$ into the join}
	
	\begin{defn}[$(\P,C)$-essential to $\neg\psi$ over $\Phi_0$] 
		Let $\Phi_0$ be a finite use-monotone Turing functional and 
			\[
				\psi = \bigwedge_i (\forall \vec{u}_i) \theta_i(\vec{u}_i)
			\]
		a $\Pi_\alpha^r(C)$ sentence in $\infinitaryL$ for some $\alpha > 0$. 
		
		For $\btau = (\tau_1, \ldots, \tau_n)$, a sequence of elements of $2^{<\omega}$ all of the same length, we say that $\btau$ is \textbf{$(\P,C)$-essential to $\neg\psi$ over $\Phi_0$}, if, for all $p\in\P$, all $i$, and all $\m$ of the correct length, if $p$ is a condition such that $p<_\P (\Phi_0,\emptyset)$, and $p\forces_\P \neg \theta_i(\m)$, then, $\Phi_p \setminus \Phi_0$ includes a triple $\ang{x,y,\sigma}$ such that $\sigma$ is compatible with at least one component of $\btau$.
	\end{defn}
	
	\begin{defn}[$T_{\P,C}(\Phi_0,\psi,k)$] 
		For each finite \umtf $\Phi_0$, each $\Pi^r_\alpha(C)$ sentence $\psi$ of our forcing language with $\alpha>0$, and each natural number $k$, let $\TPC$ be the set of length $k$ vectors of binary strings all of the same length that are $(\P,C)$-essential to $\neg\psi$ over $\Phi_0$.
	\end{defn}
	
	We drop the $(\P,C)$ prefix or subscript where confusion will not arise (as in the rest of this section, where we have fixed $C$ and are only dealing with $\P$).
	
	We order $\T$ by extension on all coordinates. Being essential is closed downward in this order. This endows $\T$ with the structure of a subtree of the length $k$ vectors of binary sequences of equal length, so $\T$ is a subtree of a recursively bounded recursive tree.
	
	\begin{lem}
		Suppose that $\Phi_0$ is a finite use-monotone Turing functional, $\psi$ is a $\Pi_\alpha^r(C)$ sentence in our forcing language with $\alpha > 0$, and $k$ is a natural number.
			\begin{enumerate}
				\item If there is a size $k$ set $\X$ of reals such that $(\Phi_0,\X)\forces \psi(\Phi_G)$, then $\T$ is infinite.
				\item If $\T$ is infinite then it has an infinite path $Y$. Further, each such $Y$ is naturally identified with a size $k$ set $\X(Y)$ of reals such that $(\Phi_0,\X(Y))\forces \psi(\Phi_G)$.
			\end{enumerate}
		\label{ess_sent_easy_1}
	\end{lem}
	
	\begin{proof}
		For the first claim, let $\X$ be such a set and let $(X_1,\ldots,X_k)$ be an enumeration of it, further, let $\btau_l$ denote $(X_1\restricted l,\ldots, X_k\restricted l)$. We show that each $\btau_l\in \T$, proving the tree is infinite.
		
		Suppose $(\forall \u_i)\theta_i(\u_i)$ is one of the conjuncts that makes up $\psi$, $\m$ is the same length as $\u_i$, and $p \leq_\P(\Phi_0,\emptyset)$ forces $\neg \theta_i(\m)$, so consequently, $p\forces \neg \psi$. As $p$ forces $\neg\psi$ and $(\Phi_0,\X)$ forces $\psi$, these two conditions are incompatible in $\P$, and as $p \leq_\P(\Phi_0,\emptyset)$ there must be an axiom $\ang{x,y,\sigma}\in \Phi_p$ which is forbidden from entering any extension of $(\Phi_0,\X)$. But then, $\sigma$ is an initial segment of some $X_i\in \X$ and so, $\sigma$ is compatible with the corresponding component of $\btau_l$, showing $\btau_l$ is essential.
		
		For the second claim, suppose $\T$ is infinite. As it is a subtree of a finitely branching tree, K\"onig's lemma provides an infinite path $Y$ through it. To each such $Y$, associate $\X(Y)$, the componentwise union of the coordinates of the members of $Y$. This will be a size $k$ set of reals. We claim $(\Phi_0,\X(Y))$ forces $\psi$ for each such $Y$.
		
		Suppose $p<(\Phi_0,\emptyset)$, and there is an $i$ and an $\m$ such that $p\forces \neg\theta_i(\m)$. Then there is an axiom $\ang{x,y,\sigma}\in\Phi_p\setminus \Phi_0$ where $\sigma$ is compatible with some component of each element of $Y$. In particular, for sufficiently large elements of $Y$, $\sigma$ is an initial segment of such a component and so is an initial segment of some $X\in \X(Y)$. So then, $p$ does not extend $(\Phi_0,\X(Y))$. Therefore, no extension of $(\Phi_0,\X(Y))$ forces $\neg\theta_i(\m)$ for any $i$ and $\m$ and as each extension of $(\Phi_0,\X(Y))$ has a further extension deciding $\theta_i(\m)$, by the definition of forcing, $(\Phi_0,\X(Y))\forces \psi$.
	\end{proof}
	
	\begin{nota}
		For a set $S\subset \omega$, we say a decision procedure or property of natural numbers is $\Pi_\alpha^0(S)$ for an $S$-recursive ordinal $\alpha>0$, if the set of solutions to the procedure or the property, respectively, is co-r.e. in $S^{(\alpha)}$ for $\alpha \geq \omega$, and co-r.e. in $S^{(\alpha-1)}$ for finite $\alpha$. We say it is $\Delta_\alpha^0(S)$ if it and its complement are both $\Pi_\alpha^0(S)$. For example a $\Pi^0_1(S)$ set is co-r.e. in $S$ but a $\Pi^0_{\omega+1}(S)$ set is co-r.e. in $S^{\omega+1}$. 
	\end{nota}
	
	\begin{lem}
		For each finite \umtf  $\Phi_0$, each $\Pi_\alpha^r(C)$ sentence $\psi$ with $\alpha > 0$, and each number $k$, $\T$ is $\Pi_{\alpha}^0(C)$ uniformly in $\Phi_0$, $\psi$ and $k$.
		\label{ess_sent_easy_2}
	\end{lem}
	\begin{proof}
		
		Fix a length $k$ vector $\btau$ of binary strings of the same length. Recall that $\btau$ is essential to $\neg\psi$ over $\Phi_0$ if for all $p = \condition{p}<_\P (\Phi_0,\emptyset)$, all conjuncts $(\forall \u)\theta_i(\u)$ which make up $\psi$, and all $\m$ of the correct length, if $p\forces \neg \theta_i(\m)$, then $\Phi_p \setminus \Phi_0$ contains an axiom $\ang{x,y,\sigma}$ where $\sigma$ is compatible with some component of $\btau$.
		
		We can express this, equivalently, as 
			\begin{quote}
				For all $\Phi_p \leq_\P\Phi_0$, all $\m$ of the correct length, and all conjuncts $(\forall \u)\theta_i(\u)$ which make up $\psi$, if there exists a set $\X$ such that $(\Phi_p,\X)\forces \neg \theta_i(\m)$, then $\Phi_p \setminus \Phi_0$ contains an axiom $\ang{x,y,\sigma}$ where $\sigma$ is compatible with some component of $\btau$.
			\end{quote} 
		Or more compactly as
			\begin{align*}
				(\forall \Phi_p)(\forall \m )(\forall i)& [ ((\Phi_p,\emptyset)\leq_\P (\Phi_0,\emptyset) \ \& \ (\exists\X)(\Phi_p,\X)\forces \neg\theta_i(\m))\\
				& \rightarrow ((\exists \ang{x,y,\sigma}\in \Phi_p\setminus \Phi_0) \text{ with $\sigma$ compatable with some $\tau\in\btau$})].
			\end{align*}
		We can also "quantify out" the size of the set $\X$, as
			\begin{align*}
				(\forall \Phi_p)(\forall \m )(\forall j)(\forall i)& [ ((\Phi_p,\emptyset)\leq_\P (\Phi_0,\emptyset) \ \& \ (\exists\X,|\X| = j)(\Phi_p,\X)\forces \neg\theta_i(\m))\\
				& \rightarrow ((\exists \ang{x,y,\sigma}\in \Phi_p\setminus \Phi_0) \text{ with $\sigma$ compatable with some $\tau\in\btau$})].
			\end{align*}
		We use both of these formulations in an inductive proof of this lemma, the first in the base case, and the second in the inductive steps. Note that in both cases, the $(\forall i)$ is a quantifier over a $C$-r.e. set: the conjuncts making up $\psi$. This is still equivalent to a single universal quantifier.
		
		Suppose $\alpha = 1$. In this case, for each $i$ and $\m$ of the correct length $\neg\theta_i(\m)$ is a quantifier-free sentence. So, by Lemma \ref{quantifierfreeforcing}, deciding whether there is an $\X$ such that $(\Phi_p,\X)\forces \neg\theta_i(\m)$ is uniformly $C$-recursive in $\Phi_p,\theta_i$ and $\m$. Also, checking whether $\Phi_p \leq_\P\Phi_0$ and whether there is an axiom $\ang{x,y,\sigma}\in \Phi_p\setminus \Phi_0$ compatible with some $\tau\in\btau$ is uniformly recursive in $\Phi_0,\Phi_p$ and $\btau$.
		
		Consequently, in the case $\alpha = 1$, $\btau$ being essential to $\neg\theta$ over $\Phi_0$ can be expressed as a property with two universal natural number quantifiers, then a universal quantifier over a $C$-r.e. set, then a matrix which is $C$-recursive uniformly in all parameters. Such an expression is co-r.e. in $C$, or $\Pi_1^0(C)$ as required.
		
		If $\alpha > 1$ is a successor, we can assume all its conjuncts are $\Sigma^r_{\alpha-1}(C)$ by Ash and Knight. By Lemma \ref{ess_sent_easy_1}, there being an $\X$ of size $j$ such that $(\Phi_p,\X)\forces \neg\theta_{i}(\m)$ is equivalent to $T(\Phi_p,\neg\theta_{i}(\m),j)$ being infinite. But $T(\Phi_p,\neg\theta_{i}(\m),j)$ is, by induction, a $\Pi_{\alpha}^0(C)$ subtree of a recursively bounded recursive tree. Hence, it being infinite is also a $\Pi_{\alpha}^0(C)$ fact, given uniformly in $\Phi_p,j,\theta_i$ and $\m$. (This bound follows because the tree being finite is equivalent to it being of bounded height and our trees are subtrees of recursive trees. Hence, saying our trees is finite is equivalent to saying there is a level of a recursive tree, from which our tree is disjoint i.e. $(\exists n)$[for all nodes $x$ in the recursive tree $T$ at level $n$, $x$ is not in our tree]. If our tree is $\Pi_\alpha^0(C)$ for some $C$ this formula is equivalent to a $\Sigma_\alpha^0(C)$ formula, and so its negation is $\Pi_\alpha^0(C)$ as required.)
		
		Thus, for successor $\alpha$, $\btau$ being essential to $\neg\psi$ over $\Phi_0$ is equivalent to a property which we can express with a block of universal natural number quantifiers, then a universal quantifier over a $C$-r.e. set, then a $\Sigma_{\alpha-1}^0(C)$ matrix uniformly in the parameters. This renders the property $\Pi_{\alpha}^0(C)$ as required.
		
		For the limit case the argument is similar. The same formulation of the definition of $\btau$ being essential works as for the successor case. However, here we have the full power of $C^{(\alpha)}$ which can decide whether recursively branching $C^{(\beta)}$-recursive trees are infinite, for $\beta < \alpha$, uniformly in an index for the tree, so, $C^{(\alpha)}$ can uniformly decide the matrix. The prefix is a block of universal quantifiers, then a universal quantifier over a $C$-r.e. set, with uniformly $C^{(\alpha)}$-recursive matrix, so $\btau$ being essential is uniformly co-r.e. in $C^{(\alpha)}$, or $\Pi_{\alpha}^0(C)$ as required.
	\end{proof}
	\begin{cor}
		Suppose that $S\subseteq \omega$ is not $\Delta_{\alpha}^0(C)$. Let $\Phi_0$ be a finite use-monotone Turing functional, $\psi$ a $\Pi_\alpha^r(C)$ sentence with $\alpha>0$, and $k\geq 1$. If there is a size $k$ set $\X$ of reals such that $(\Phi_0,\X)\forces \psi$, then there is such a set not containing $S$. Moreover, we can find such an $\X$ all of whose members are recursive in $C^{(\alpha+1)}$ uniformly in $\psi,k$ and $S\oplus C^{(\alpha+1)}$.
		\label{ess_sent_easy_3}
	\end{cor}
	\begin{proof}
		Suppose there is a size $k$ set $\X$ such that $(\Phi_0,\X)\forces \psi$. By Lemmas \ref{ess_sent_easy_1} and \ref{ess_sent_easy_2} the tree $\T$ is a $\Pi_{\alpha}^0(C)$ subtree of a recursively bounded recursive tree which has an infinite path. The paths form a nonempty $\Pi^0_{\alpha}(C)$ class and so there is a path $Y$ in which $S$ is not recursive. Then, certainly $S$ is not a member of $\X(Y)$. Further, Lemma \ref{ess_sent_easy_1} implies $(\Phi_0,\X(Y))\forces \psi$. This completes the proof of the first claim.
		
		For the second claim, we need to find an infinite path in $\T$ on which $S$ does not appear, recursively in $C^{(\alpha+1)}\oplus S$. We need $C^{(\alpha+1)}$ to decide which members of $\T$ have an infinite part of $\T$ above it, and $S$ to decide whether a given node could possibly have an extension on which $S$ appears. Consequently, there is an algorithm, recursive in $S\oplus C^{(\alpha + 1)}$, solving this problem which, searches through the tree of length $k$ sequences of binary strings of the same length, looking for one which disagrees with $S$ and has an infinite part of $\T$ above it. Once one is found, the algorithm switches to searching for any path in $\T$ through this node. Note the uniformity in $\psi$ and $k$ follows from the uniformity in the construction of $\T$.
		
		Of course, once we have found a node each of the components of which are not compatible with $S$, yet has a path through it, we never use $S$ again in the construction. Hence, we could externally hard code to only search for paths in $\T$ through some fixed node, and any such path $Y$ corresponds to reals $\X(Y)$ which are recursive in $C^{(\alpha+1)}$ (as $C^{(\alpha+1)}$ is constructing such a $Y$), hence, the members of the $\X$ we build are $C^{(\alpha+1)}$-recursive although we lose the uniformity by externally hard-coding in some information.
	\end{proof}
	
	\begin{cor}
		Suppose that $\alpha > 0$, $S$ is not $\Delta_{\alpha}^0(C)$ and $\psi$ is a $\Pi_\alpha^r(C)$ sentence. For any condition $p = (\Phi_p,\X_p)$ with $S\notin \X_p$, there is a stronger $q$ which we can find uniformly in $\Phi_p$ and $S\oplus C^{(\alpha+1)}\oplus \X_p$ such that
		\begin{enumerate}
			\item $S\notin \X_q$ and each $X\in \X_q \setminus \X_p$ is recursive in $C^{(\alpha+1)}$.
			\item For all $x$, if $\Phi_q(S)(x)$ is defined, then $\Phi_p(S)(x)$ is defined. That is, $q$ does not add any new computations which apply to $S$.
			\item	Either $q \forces \psi$ or there is a conjunct $\theta_i$ and an $\m$ such that $q\forces \neg\theta_i(\m)$ (and we can tell which formula we have forced).
		\end{enumerate}
		\label{ess_sent_easy_4}
	\end{cor}
	\begin{proof}
		Fix a condition $p = \condition{p}$, and let $k = |\X_p|$. We would like to know whether $T(\Phi_0,\psi,k+1)$ is infinite. It is a $\Pi^0_{\alpha}(C)$ subtree of a recursively bounded recursive tree, and hence, this is a $\Pi^0_{\alpha}(C)$ fact. If $T(\Phi_0,\psi,k+1)$ is infinite, then Corollary \ref{ess_sent_easy_3} provides a path $Y$ through it, not containing $S$, each member of which is recursive in $S\oplus C^{(\alpha+1)}$. Then, $(\Phi_p,\X_p \cup \X(Y))\leq_\P (\Phi_p,\X)$, $(\Phi_p,\X_p \cup \X(Y))\forces \psi$, and, clearly, (2) is satisfied, as we haven't changed the finite part.
		
		If $T(\Phi_0,\psi,k+1)$ is not infinite, then (enumerating $\X = \{X_1,\ldots,X_k\}$) $(X_1,\ldots,X_k,S)$ does not provide a path through it. Consequently, for some $l$, $(X_1\restricted l,\ldots,X_n\restricted l, S\restricted l)$ is not essential to $\neg\psi$ over $\Phi_p$, meaning there is a conjunct $\theta_i$, numerals $\m$ and a condition $q=\condition{q}$ extending $(\Phi_p,\emptyset)$ such that $q\forces_\P \neg\theta_i(\m)$ yet, every new axiom $\ang{x,y,\sigma}\in\Phi_q\setminus\Phi_p$ is incompatible with each coordinate of $(X_1\restricted l,\ldots,X_n\restricted l, S\restricted l)$. In particular, there are no new axioms applying to $S$, nor to any $X_i\in \X$, and by Lemma \ref{ess_sent_easy_1} there is a $j$ such that $T(\Phi_q,\neg\theta_i(\m),j)$ is infinite.
		
		We can find which $\Phi_q<\Phi_p$ add neither any new computations applying to $S$, nor to any member of $\X_p$, recursively in $S\oplus \X_p$, and then, whether $T(\Phi_q,\neg\theta_i(\m),j)$ is infinite recursively in $C^{(\alpha)}$. Consequently, we can find such a $\Phi_q<\Phi_p,\theta_i,\m$ and $j$ recursively in $S\oplus C^{(\alpha+1)}\oplus \X_p$, (we add the extra jump because for each $\Phi_q < \Phi_p$ we ask whether there exists a $\theta_i,\m$ and $j$ such that $T(\Phi_p,\neg\theta_i(\m),j)$ is infinite). Once we have found such a $\Phi_q,\theta_i,\m$ and $j$, applying Corollary \ref{ess_sent_easy_4} again, there is a size $j$ set $\X$ with $S\notin \X$ such that every $X\in\X$ is uniformly recursive in $S\oplus C^{(\alpha+1)}$ and $(\Phi_q,\X)\forces \neg\theta_i(\m)$, therefore, $(\Phi_q,\X\cup\X_p)$ is the condition we want.
	\end{proof}
	
	As $A$ is not $\Delta_\alpha^0(C)$ for any $C$-recursive $\alpha$, we can repeatedly apply Corollary \ref{ess_sent_easy_4} to construct an $\infinitaryL$-generic sequence in $\P$, without adding $A$ to the infinite part of any condition, and without adding any axioms applying to $A$. Hence, we can dovetail this construction with one which adds axioms $\ang{n,B(n),\alpha}$ for some sufficiently long $\alpha \subset A$ which effectively codes $B$ into the join of $A$ and $\Phi_G$. Now we proceed with showing that we can diagonalize against $C\oplus \Phi_G$ computing $D$.
	
	\subsection*{Diagonalizing against $D$}
	
	\begin{nota}
		For a real $S$, we say $\ang{e,\alpha}$ is an $S$-hyperarithmetic reduction, if $e$ is an index and $\alpha$ is an $S$-recursive ordinal. We write $\ang{e,\alpha}^S$ for the partial function $\{e\}^{S^{(\alpha)}}$.
	\end{nota}

	\begin{defn}[$(\P,C)$-Essential to splits]
		Let $\Phi_0$ be a finite \umtf and $\ang{e,\alpha}$ a $C$-hyperarithmetic reduction. For each $\btau$, a finite sequence of binary strings of the same length, we say $\btau$ is \textbf{$(\P,C)$-essential to $\ang{e,\alpha}$-splits below $\Phi_0$} if whenever $p,q<_\P (\Phi_0,\emptyset)$ form an $\ang{e,\alpha}$-split (i.e. there is an $x$ such that $p\forces \ang{e,\alpha}^{C\oplus \Phi_G}(x)\downarrow = k_1,q\forces \ang{e,\alpha}^{C\oplus\Phi_G}(x)\downarrow = k_2$ and $k_1\neq k_2$) there is some new axiom $\ang{x,y,\sigma}\in\Phi_p \cup \Phi_q$ but not in $\Phi_0$ such that $\sigma$ is compatible with some component of $\btau$.
	\end{defn}
	We let $\UPC$ be the set of $\btau$ of length $k$ which are $(\P,C)$-essential to $\ang{e,\alpha}$-splits below $\Phi_0$, and consider this as a tree as before. Also, as before, we drop the $\P$ and $C$ when confusion won't arise.
	
	\begin{lem}
		Let $\Phi_0$ be a Turing functional, $\ang{e,\alpha}$ a $C$-hyperarithmetic reduction, and $k$ a natural number.
		\begin{itemize}
			\item If $\X$ is a size $k$ set of reals such that no pair $p,q<_\P (\Phi_0,\X)$ form an $\ang{e,\alpha}$-split, then $\U$ is infinite.
			\item If $\U$ is infinite then it has an infinite path, and each such path $Y$ is identified with a size $k$ set $\X(Y)$ such that $(\Phi_0,\X(Y))$ has no $\ang{e,\alpha}$-splits below it.
		\end{itemize}
		\label{ess_splits_easy_1}
	\end{lem}
	
	\begin{proof}
		Let $\btau_l = (X_1\restricted l, \ldots, X_k\restricted l)$ where $X_1,\ldots,X_k$ enumerates $\X$. Suppose $p,q<_\P (\Phi_0,\emptyset)$ form an $\ang{e,\alpha}$-split. As $(\Phi_0,\X)$ has no splits below it at least one of $p$ and $q$ are not compatible with $(\Phi_0,\X)$. Suppose it is $p$ which is incompatible. As $p<_\P (\Phi_0,\emptyset)$, $p$ must have a new axiom $\ang{x,y,\sigma}$ which some $X_i$ forbids. Hence, $\sigma \subset X_i$ and, therefore, $\sigma$ is compatible with the $i$th component of $\btau_l$. Hence, every $\btau_l$ is essential to $\ang{e,\alpha}$-splits, and so $\U$ is infinite.
		
		For the second claim, suppose $\U$ is infinite. Then, Konig's lemma provides us a path $Y$, and $\X(Y)$ is the componentwise union of $Y$. Suppose $p,q\leq_\P(\Phi_0,\X(Y))$ form an $\ang{e,\alpha}$-split below $(\Phi_0,\X(Y))$, then they also form a split below $(\Phi_0,\emptyset)$. As the elements of $Y$ are essential, there is some new axiom $\ang{x,y,\sigma}$ in (WLOG) $\Phi_p$, but not in $\Phi_0$ compatible with some component of each element of $Y$. Hence, $\sigma$ is an initial segment of some component of $\X(Y)$, and so $p\nleq_\P (\Phi_0,\X(Y))$, a contradiction.
	\end{proof}
	
	\begin{lem}
		If $\Phi_0$ is a Turing functional, $\ang{e,\alpha}$ a $C$-hyperarithmetic reduction, and $k$ a natural number then, $\U$ is $\Pi^0_{\alpha+3}(C)$ uniformly in $\Phi_0$, $e$, and $k$.
		\label{ess_splits_easy_2}
	\end{lem}
	\begin{proof}
		Fix $\Phi_0,\ang{e,\alpha},k$, and $\btau$ of length $k$. Then $\btau \in \U$ if and only if $\btau$ is essential to splits below $\Phi_0$, which is equivalent to:
		\begin{align*}
		(\forall p,q<_\P (\Phi_0,\emptyset))(\forall m,m_1, m_2)
		& [ (p\forces \ang{e,\alpha}^{(C\oplus \Phi_G)}(m)\downarrow = m_1\\ 
		&\land \ q\forces \ang{e,\alpha}^{(C\oplus \Phi_G)}(m) \downarrow = m_2 \\
		&\land m_1\neq m_2) \\
		& \rightarrow (\exists\ang{x,y,\sigma})\text{ such that }\ang{x,y,\sigma}\in (\Phi_p\cup \Phi_q)\setminus \Phi_0 \\ 
		& \hspace{8mm} \text{yet $\sigma$ compatible with some }\tau\in\btau)]
		\end{align*}
		Using a similar trick as in Lemma \ref{ess_sent_easy_2}, we rewrite this as
		\begin{align*}
		(\forall \Phi_p,\Phi_q\leq_\P \Phi_0)(\forall m,m_1, m_2)
		& [ ((\exists \X_p)(\Phi_p,\X_p)\forces \ang{e,\alpha}^{(C\oplus \Phi_G)}(m)\downarrow = m_1\\ 
		&\land (\exists \X_q)(\Phi_q,\X_q) \forces \ang{e,\alpha}^{(C\oplus \Phi_G)}(m) \downarrow = m_2 \\ 
		&\land m_1\neq m_2) \\
		& \rightarrow (\exists\ang{x,y,\sigma})\text{ such that }\ang{x,y,\sigma}\in (\Phi_p\cup \Phi_q)\setminus \Phi_0 \\ 
		& \hspace{8mm} \text{yet $\sigma$ compatible with some }\tau\in\btau)]
		\end{align*}
		But, by Lemma \ref{ess_sent_easy_1}, we can replace
		\[
		(\exists \X_p)\left[(\Phi_p,\X_p)\forces \ang{e,\alpha}^{(C\oplus \Phi_G)}(m)\downarrow = m_1\right]
		\]
		with 
		\[
		(\exists j)\left[T(\Phi_p,\ang{e,\alpha}^{C\oplus \Phi_G}(m)\downarrow = m_1,j)\text{ is infinite}\right]
		\]
		and make similar replacements for $\Phi_q$ and $\X_q$.
		
		Now, "$ \ang{e,\alpha}^{(C\oplus \Phi_G)}(m)\downarrow = m_1$" can be expressed as a $\Sigma_{\alpha+1}^r(C)$ sentence in our forcing language, uniformly in $m,m_1, e,\alpha$. So then, we would like to know whether a $\Pi_{\alpha+2}^0(C)$ subtree of a recursively bounded recursive tree (have to go up one level one to get a $\Pi^r$ sentence) is infinite, which is a $\Pi_{\alpha+2}^0(C)$ condition, uniformly in $j,m,m_1,\Phi_p$. Hence, $\btau$ being essential to $\ang{e,\alpha}$-splits below $\Phi_0$ is equivalent to a sentence with an initial block of universals, with a conditional matrix, the antecedent of which is $\Pi^0_{\alpha+2}(C)$ and the consequent uniformly recursive. Hence, $\btau$ being essential to $\ang{e,\alpha}$-splits below $\Phi_0$ is a $\Pi_{\alpha+3}(C)$ property. Uniformity is clear from the analysis.
	\end{proof}
	
	\begin{cor}
		For each finite \umtf $\Phi_0$, each $C$-hyperarithmetic reduction $\ang{e,\alpha}$, and each natural number $k$, if $S$ is not $\Delta^0_{\alpha+3}(C)$ and if there is a size $k$ set $\X$ such that $(\Phi_0,\X)$ has no $\ang{e,\alpha}$-splits below it, then there is such an $\X$ not containing $S$. Further, we can find such an $\X$ all of whose members are recursive in $C^{(\alpha+4)}$ indeed, we can construct $\X$ uniformly in $e,k$ and $S\oplus C^{(\alpha+4)}$
		\label{ess_splits_easy_3}
	\end{cor}
	
	\begin{proof}
		Suppose there is a size $k$ set $\X$ such that $(\Phi_0,\X)$ has no $\ang{e,\alpha}$-splits. Then, by Lemma \ref{ess_splits_easy_1}, $\U$ is infinite, and by Lemma $\ref{ess_splits_easy_2}$ it is uniformly $\Pi_{\alpha + 3}^0(C)$. By the same argument as for Corollary \ref{ess_sent_easy_3}, this tree has a path, $Y$, recursive in $C^{(\alpha + 4)}$, and yet, $S$ is not recursive in the members of the path, and consequently, $S$ is not on the path. Furthermore, by Lemma \ref{ess_splits_easy_1}, the reals $\X(Y)$ serve as the second coordinate of a condition $(\Phi_0,\X(Y))$, which has no $\ang{e,\alpha}$-splits below it.
		
		Now, for the uniform construction, we need to build a path through $\U$, a $\Pi^0_{\alpha + 3}(C)$ recursively bounded tree, and we need to make sure $S$ is not on the path. $C^{(\alpha + 4)}$ suffices to decide which members of the tree have an infinite part of the tree above them. Additionally $S$ can check whether a given node has coordinates all of which disagree with $S$, and as $\U$ is given uniformly in the mentioned parameters, we have the desired uniformity.
	\end{proof}
	
	\begin{cor}
		Suppose $p\in \P$ is a forcing condition such that $A\notin \X_p$, and $\ang{e,\alpha}$ a $C$-hyperarithmetic reduction. Then, there is an extension of $p$ which adds no new axioms applying to $A$ and does not have $A$ in the infinite part, which diagonalizes against $\ang{e,\alpha}^{C\oplus \Phi_G} = D$ i.e. either forces $\ang{e,\alpha}^{C\oplus\Phi_G}$ is not total or there is an $x$ such that it forces this computation on $x$ to be convergent, yet not equal to $D(x)$.
		\label{diagonalising_easy}
	\end{cor}
	
	\begin{proof}
		Fix $p$. The expression "$\ang{e,\alpha}^{C\oplus \Phi_G}$ is total" can be rendered in our forcing language as a $\Pi_{\alpha+2}^r(C)$ sentence. Then, by Corollary \ref{ess_sent_easy_4}, we can find a stronger condition $q$ neither adding $A$ to the infinite part nor any new axioms applying to $A$ to the finite part, deciding this sentence. If it forces it to be false (i.e. the reduction is not total), $q$ is the extension we want.
		
		Otherwise suppose $q$ forces totality. If $q$ has an $\ang{e,\alpha}^{C\oplus \Phi_G}$ split, then there is an $r\leq_\P q$ and an $x$ such that $r\forces \ang{e,\alpha}^{C\oplus \Phi_G}(x)\downarrow \neq D(x)$. As there is an $r\leq_\P q$ forcing this sentence, then by Corollary \ref{ess_sent_easy_3} there is such an $r$ which neither adds $A$ to the infinite part, nor adds any new computations applying to $A$ to the finite part. Such an $r$ is the extension we want.
		
		Now suppose $q$ has no such splits. Then, $\ang{e,\alpha}^{C\oplus \Phi_G}$ is determined by $q$. We claim we can hyperarithmetically in $C$ determine what set $q$ determines this to be, and so compute $\ang{e,\alpha}^{C\oplus \Phi_G}$ (as determined by $q$) hyperarithmetically in $C$, which implies it is not $D$.
		
		To see this, recall that as $q=\condition{q}$ has no splits, then there is a finite set of reals $\X$ (still not containing $A$) all of whose members are recursive in $C^{(\alpha+4)}$ such that $q'=(\Phi_q,\X)$ also has no splits, by Corollary \ref{ess_splits_easy_3}. Note $q'$ and $q$ are compatible.
		
		Now, for fixed $x$ and each $y$ and $k$, search for an $\Phi_r$ such that $(\Phi_r,\X) \leq_\P (\Phi_{q},\X)$, and $T(\Phi_r,\ang{e,\alpha}^{C\oplus \Phi_G}(x)\downarrow = y,k)$ is infinite. Checking the first condition is recursive in $C^{(\alpha+4)}$, and so too is the second, as, the tree is $\Pi_{\alpha+2}^0(C)$, uniformly in $\Phi_r,e,\alpha,x,y,k$, hence, this search is $C^{(\alpha+4)}$-recursive. Once we find such a $y,k$ we can output $y$ as the value of $\ang{e,\alpha}^{C\oplus \Phi_G}(x)$. As $q'$ and $q$ are compatible and $q'$ has no splits, $y$ must be the value that $q$ decides for $\ang{e,\alpha}^{C\oplus \Phi_G}(x)$. We know such $\Phi_r,y$, and $k$ exist, because if they didn't, then we could force nonconvergence of the computation at $x$, contrary to hypothesis.
			
		Thus, recursively in $C^{(\alpha+3)}$ we can compute what $q$ forces $\ang{e,\alpha}^{C\oplus \Phi_G}$ to be, and so, as $D$ is not hyperarithmetic in $C$, this computation does not compute $D$.
	\end{proof}
	
	So, we have shown that we can construct a generic sequence for the easy case, while still effecting the coding. We will postpone the proof of preservation of $\omega_1^{C}$ until we have proved results analogous to the above for the hard case.
	
\section{The hard case}
		In the previous section, we proved results which we will apply to $A$ and $C$, provided $A$ is not hyperarithmetic in $C$ i.e. if we are in the easy case. In the hard case, these lemmas no longer apply. This is the motivation for the forcing $\Q_{A,B}$ (which we abbreviate as $\Q$ in this section), where we build into the notion of forcing the lemmas that we can't push through for $\P$. Recall that, in the hard case, we choose representatives such that $A\leq_T C$. In this section when we say a condition forces a sentence, we mean $\forces_\Q$.
		
		Note that if we are successful in our goal of coding $B$ into the join of $A$ and $\Phi_G$, then in the hard case we will have
			\[
				C\oplus \Phi_G \geq_T A \oplus \Phi_G \geq_T B
			\]
		and, consequently, $C\oplus \Phi_G \geq_T B \oplus C$. Thus, we will attempt to preserve the fact that $B\oplus C\ngeq_h D$ in the hard case. This requires us to use the language $\infinitaryLB$ as our language of forcing for $\Q$.
		
		\begin{defn}[$(\Q,B\oplus C)$-essential to $\neg\psi$ over $\Phi_0$]
			If $\Phi_0$ is a finite \umtf which partially computes $B$ given $A$, $\psi$ is a $\Pi^r_\alpha(B\oplus C)$ sentence in $\infinitaryLB$, and $\btau$ a finite tuple of binary strings, all of the same length, we say $\btau$ is $(\Q,B\oplus C)$-essential to $\neg\psi$ over $\Phi_0$, if, whenever $q = \condition{q}\in \Q$ properly extends $(\Phi_0,\emptyset)$ and  $q\forces_\Q\neg\theta_i(\m)$ for one of the conjuncts $(\forall \u) \theta_i(\u)$ making up $\psi$ and some tuple $\m$ of the correct length, $\Phi_q\setminus\Phi_0$ contains an axiom $\ang{x,y,\sigma}$ such that $\sigma$ is compatible with some $\tau\in\btau$.
		\end{defn}
	We let $\TQC$ be the tree of essential tuples of length $k$, as before. We will abuse notation and omit the subscripts of $\T$. Again, $\T$ is a subtree of a recursively bounded recursive tree. Note that $\T$ could have branches which contain $A$ as the limit of one of the coordinates, meaning we can't use that branch as the set of reals for a $\Q$ condition. Given a path $Y$ in the tree $\T$ (or any of our trees) and a real $S$ we say \textbf{$S$ does not appear on $Y$} if $S$ is not the limit of any of the coordinates of $Y$.
	
	\begin{lem}
		Suppose that $\Phi_0$ is a finite \umtf which partially computes $B$ on input $A$, $\psi$ is a $\Pi_\alpha^r(B\oplus C)$ sentence with $\alpha > 0$ and $k$ is a natural number.
		\begin{enumerate}
			\item If there is a size $k$ set $\X$ of reals such that $(\Phi_0,\X)\in\Q$ and forces $\psi(\Phi_G)$, then $\T$ has an infinite path on which $A$ does not appear.
			\item If $\T$ has an infinite path on which $A$ does not appear, then each such path $Y$ is naturally identified with a size $k$ set $\X(Y)$ (not containing $A$) of reals such that $(\Phi_0,\X(Y))\in\Q$ and forces $\psi(\Phi_G)$.
		\end{enumerate}
		\label{ess_sent_hard_1}
	\end{lem}
	
	\begin{proof}
		The proof is similar to that of Lemma \ref{ess_sent_easy_1}. For the first claim, note that, for each $l$, $\btau_l=(X_1\restricted l,\ldots,X_k\restricted l)$ is essential (where $(X_1,\ldots,X_k)$ enumerates such an $\X$), as, any $q\in \Q$ extending $(\Phi_0,\emptyset)$ and forcing $\neg \theta_i(\m)$ for some conjunct $(\forall \u)\theta_i(\u)$ making up $\psi$ and some $\m$, must be incompatible with $(\Phi_0,\X)$. Consequently, there is a new axiom in $\Phi_q$ applying to some member of $\X$ and so compatible with the corresponding component of $\btau_l$. Clearly, $A$ does not appear on the path because $(\Phi_0,\X)\in \Q$.
		
		For the second claim, for any path $Y$ in $\T$ on which $A$ does not appear, $\X(Y)$ - the coordinatewise union of $Y$ - may serve as the infinite part of a $\Q$ condition. Then, as in Lemma \ref{ess_sent_easy_1}, if $q<_\Q(\Phi_0,\emptyset)$ forces $\neg\theta_i(\m)$ for a $\theta_i$ and an $\m$ as above, then $q$ must contain a new axiom compatible with some coordinate of each element of $Y$, and so, must apply to some $X\in \X(Y)$. Consequently, such a $q$ does not extend $(\Phi_0,\X)$ and so $(\Phi_0,\X)\forces_Q \psi$.
	\end{proof}
	
	\begin{lem}
		For each finite \umtf $\Phi_0$ that correctly computes $B$ on input $A$, each $\Pi_\alpha^r(B\oplus C)$ sentence $\psi$ in $\infinitaryLB$ with $\alpha > 0$, and each number $k$, $\T$ is $\Pi_{\alpha}^0(B\oplus C)$ uniformly in $\Phi_0$, $\psi$, and $k$.
		\label{ess_sent_hard_2}
	\end{lem}
	\begin{proof}
		We use the same strategy as for Lemma \ref{ess_sent_easy_2}. Fix $\btau$ of length $k$, $\Phi_0$ and $\psi$. Then, $\btau$ being essential to $\neg\psi$ over $\Phi_0$ is equivalent to both of the following
		\begin{align*}
		(\forall \Phi_q)(\forall \m )(\forall i)& [ ((\Phi_q,\emptyset)\leq_\Q (\Phi_0,\emptyset) \ \& \ (\exists\X)(\Phi_q,\X)\forces_\Q \neg\theta_i(\m))\\
		& \rightarrow (\exists \ang{x,y,\sigma})\text{ such that }\ang{x,y,\sigma}\in \Phi_q\setminus \Phi_0 \\
		&\hspace{8mm} \text{yet $\sigma$ is compatable with some $\tau\in\btau$})], 
		\end{align*}
		and
		\begin{align*}
		(\forall \Phi_q)(\forall \m )(\forall j)(\forall i)& [ ((\Phi_q,\emptyset)\leq_\Q (\Phi_0,\emptyset) \ \& \ (\exists\X,|\X| = j)(\Phi_q,\X)\forces_Q \neg\theta_i(\m))\\
		& \rightarrow (\exists \ang{x,y,\sigma})\text{ such that }\ang{x,y,\sigma}\in \Phi_q\setminus \Phi_0 \\
		& \hspace{8mm}\text{yet $\sigma$ is compatable with some $\tau\in\btau$})].
		\end{align*}
		(Again the $(\forall i)$ is a universal quantifier over a $(B\oplus C)$-r.e. set.) 
		
		Now suppose $\alpha = 1$, then by Lemma \ref{quantifierfreeforcing}, the antecedent in the first formulation is $A\oplus B\oplus C$-recursive uniformly in $\Phi_0,\Phi_q, \theta_i$ and $\m$. As $C\geq_T A$, we can ignore $A$, and then $\btau$ being essential can be rendered as a property with two universal natural number quantifiers, then a universal quantifier over a $(B\oplus C)$-r.e. set, then a $B\oplus C$-recursive matrix, uniformly in the parameters. Hence, this is a $\Pi_1^0(B\oplus C)$ property.
		
		For the inductive steps, firstly, suppose $\alpha > 1$ is a successor. We now consider the second equivalent definition of $\btau$ being essential. By Lemma \ref{ess_sent_hard_1}, "$(\exists \X,|\X|=j)(\Phi_q,\X)\forces \neg\theta_i(\m)$" is equivalent to $T(\Phi_q,\neg\theta_i(\m),j)$ having a path on which $A$ does not appear. We may render this as: Does there exists a length $l$ so that for every length $l'>l$ there exists a node on the tree at level $l'$ each coordinate of which is not an initial segment of $A$? Because $T(\Phi_q,\neg\theta_i(\m),j)$ is a subtree of a recursively bounded recursive tree, the second existential quantifier is actually a bounded existential quantifier. As, by induction, $T(\Phi_q,\neg\theta_i(\m),j)$ is $\Pi^0_{\alpha-1}(B\oplus C)$ and $A\leq_T C$ so the tree having a path on which $A$ does not appear is $\Sigma^0_{\alpha}(B\oplus C)$. 
		
		As this is the antecedent of a conditional, we flip to $\Pi^0_{\alpha}(B\oplus C)$ and then our property is expressed via an initial block of universal quantifiers, and then a $\Pi^0_{\alpha}(B\oplus C)$ matrix, for $\Pi^0_{\alpha}(B\oplus C)$ as required.
		
		Finally, suppose $\alpha>1$ is a limit. We use the second formulation of $\btau$ being essential. Now, by induction, $(B\oplus C)^{(\alpha)}$ can (recursively) decide whether $T(\Phi_p,\neg\theta_i(\m), j)$ has an infinite path on which $A$ does not appear, for any conjunct $(\forall \u)\theta_i(\u)$ of $\psi$. Consequently, when $\alpha$ is a limit, the matrix of the second formulation for $\btau$ being essential is recursive in $(B\oplus C)^{(\alpha)}$, and the prefix is a block of universals for $\Pi_{\alpha}^0(B\oplus C)$ as required.
	\end{proof}
	
	\begin{cor}
		Let $\Phi_0$ be a finite use-monotone Turing functional, $\psi$ a $\Pi^r_\alpha(B\oplus C)$ sentence, and $k$ a natural number. If there is a size $k$ set $\X$ such that $A\notin X$ and $(\Phi_0,\X)\forces_\Q \psi$, then we can - recursively in $(B\oplus C)^{(\alpha+1)}$ - find such an $\X$, uniformly in $\psi,k$ (of course, such a set does not contain $A$).
		\label{ess_sent_hard_3}
	\end{cor}
	
	\begin{proof}
		If there is such a set $\X$, then by Lemma \ref{ess_sent_hard_1}, $\T$ has an infinite path on which $A$ does not appear. By Lemma \ref{ess_sent_hard_2}, $\T$ is a $\Pi^0_{\alpha}(B\oplus C)$ subtree of a recursively bounded recursive tree. Consequently, $(B\oplus C)^{(\alpha+1)}$ can decide membership in $\T$, can decide whether the tree is infinite above any given node, and can decide whether any node is incompatible with $A$. Thus, $(B\oplus C)^{(\alpha+1)}$ can construct a path $Y$ in $\T$ on which $A$ does not appear, and then, $(\Phi_0,\X(Y))\forces_\Q \psi$. 
	\end{proof}
	
	\begin{cor}
		Let $\psi$ be a $\Pi^r_\alpha(B\oplus C)$ sentence and $k$ a natural number. For any condition $q\in \Q$ there is a stronger $r\leq_\Q q$ which we can find uniformly in $\Phi_q$ and $(B\oplus C)^{(\alpha+2)}\oplus \X_q$ such that every new $X\in \X_r$ is recursive in $(B\oplus C)^{(\alpha+1)}$, and $r$ decides $\psi$ (and we know whether $r\forces_\Q \psi$ or $r\forces_\Q \neg \psi$).
		\label{ess_sent_hard_4}
	\end{cor}
	\begin{proof}
		
		Fix a $q\in \Q$ and enumerate $\X_q$ as $X_1,\ldots,X_k$. There are three cases to consider: Either $T(\Phi_q,\psi,k)$ has an infinite path on which $A$ does not appear, or $T(\Phi_q,\psi,k)$ is finite, or $T(\Phi_q,\psi,k)$ is infinite (and so has a path) but every such path contains $A$.
		
		$T(\Phi_q,\psi,k)$ is a $\Pi_{\alpha}^0(B\oplus C)$ subtree of a recursively bounded recursive tree by given uniformly in the parameters, by Lemma \ref{ess_sent_hard_2}. Hence, the question of $T(\Phi_q,\psi,k)$ being infinite is uniformly $\Pi^0_{\alpha}(B\oplus C)$ and so is decidable in $(B\oplus C)^{(\alpha +1)}$. Furthermore, if it is infinite, $(B\oplus C)^{(\alpha+2)}$ can decide whether there is a path on which $A$ does not appear, because $A$ is recursive in $C$. Therefore, $(B\oplus C)^{(\alpha+2)}$ can uniformly determine which case we are in.
		
		\emph{Case 1:} If we find $T(\Phi_q,\psi,k)$ has an infinite path on which $A$ does not appear $(B\oplus C)^{(\alpha+1)}$ suffices to construct such a path $Y$ as in Corollary \ref{ess_sent_hard_3}. By Lemma \ref{ess_sent_hard_1} $(\Phi_q,\X(Y))\forces_\Q \psi$ and so $r = (\Phi_q, \X_q\cup \X(Y))$ suffices and was found uniformly. 
		
		\emph{Case 2:} In this case $T(\Phi_q,\psi,k)$ is finite and so $(X_1,\ldots,X_k)$ does not form a path through it. Therefore, for some $l$, $\btau_l = (X_1\restricted l,\ldots,X_k\restricted l)$ is not essential to $\neg\psi$ over $\Phi_q$. Then there is some $p <_\Q (\Phi_q,\emptyset)$, some conjunct $\theta_i(\u)$ of $\psi$ and an $\m$ of the correct length such that, $p\forces_\Q \neg \theta_i(\m)$ yet no $\ang{x,y,\sigma}\in \Phi_p\setminus \Phi_q$ is compatible with any component of $\btau_l$. As $p\forces \neg \theta_i(\m)$ there is a $j$ such that $T(\Phi_p,\neg\theta_i(\m),j)$ has an infinite path on which $A$ does not appear by Lemma \ref{ess_sent_hard_1}.
		
		$(B\oplus C\oplus \X_q)$ can recursively find all $\Phi_p$ such that $\Phi_p\leq_\Q \Phi_q$ and adds no computations applying to any $X\in\X_q$. Also for fixed $\theta_i,\m,j$ the question of whether $T(\Phi_p,\neg\theta_i(\m),j)$ has an infinite path on which $A$ does not appear can be decided by $(B\oplus C)^{(\alpha+1)}$ (as $T(\Phi_p,\neg\theta_i(\m),j)$ at least one level of complexity lower than $T(\Phi_p,\psi,j)$) and so $(B\oplus C)^{(\alpha+1)}\oplus \X_q$ suffices to find such a $\Phi_p,\theta_i,\m$ and $j$ and to construct a path $Y$ in the tree, and so $r = (\Phi_p,\X(Y)\cup \X_q)$ is the condition we want.
		
		\emph{Case 3:} In the final case $T(\Phi_q,\psi,k)$ is infinite, but every path contains $A$. We claim that $q$ does not force $\psi$, because if $q\forces_\Q \psi$, then by Lemma \ref{ess_sent_hard_1}, $\X_q$ forms a branch of $T(\Phi_q,\psi,k)$ not containing $A$. As $\psi$ is a universal sentence and $q$ does not force it, there is some $p\leq_\Q q$, some conjunct $\theta_i(\u)$ and some tuple $\m$ such that for every $p'\leq_\Q p$, $p'$ does not force $\theta_i(\m)$. Therefore, $p\forces \neg \theta_i(\m)$.
		
		For such an $p$ there is a $j$ such that $T(\Phi_p,\neg\theta_i(\m),j)$ has a path $Y$ on which $A$ does not appear, and so by Corollary \ref{ess_sent_hard_3} $(B\oplus C)^{(\alpha+1)}$ can find such a $\Phi_p,\theta_i,\m$ and then construct a path $Y$, so then, the condition we want is $r =( \Phi_p,\X_q\cup \X(Y))$. 
	\end{proof}
	
	The reader may be wondering why we are proving theorems which are the analogues of theorems which prevent us from adding $A$ to the infinite part of a $\P$ condition, when all $\Q$ conditions have that built in. It is, in fact, the bounds on the complexity of the forcing relation that we want. These bounds will allow us to diagonalize against computing $D$.
	
	\begin{defn}[$(\Q,B\oplus C)$-Essential to splits]
		Let $\Phi_0$ be a finite \umtf which correctly computes $B$ on input $A$, and $\ang{e,\alpha}$ a $(B\oplus C)$-hyperarithmetic reduction. For each $\btau$, a finite sequence of binary strings of the same length, we say $\btau$ is \textbf{$(\Q,B\oplus C)$-essential to $\ang{e,\alpha}$-splits below $\Phi_0$} if whenever $q,r\leq_\Q (\Phi_0,\emptyset)$ form an $\ang{e,\alpha}$-split (i.e. there is an $x$ such that $q\forces_\Q \ang{e,\alpha}^{B\oplus C\oplus \Phi_G}(x)\downarrow = k_1,r\forces_\Q \ang{e,\alpha}^{B\oplus C\oplus\Phi_G}(x)\downarrow = k_2$ and $k_1\neq k_2$) there is some new axiom $\ang{x,y,\sigma}\in\Phi_q \cup \Phi_r$ but not in $\Phi_0$ such that $\sigma$ is compatible with some component of $\btau$.
	\end{defn}
	
	We let $\UQC$ denote the set of length $k$ vectors $\btau$ which are $(\Q,B\oplus C)$-essential to $\ang{e,\alpha}$-splits below $\Phi_0$. As with $\UPC$, this set can be considered as a subtree of a recursively bounded recursive tree. We drop the $\Q$ and $B\oplus C$ decorations wherever possible.
	
		\begin{lem}
			Let $\Phi_0$ be a Turing functional which partially computes $B$ on input $A$, $\ang{e,\alpha}$ a $B\oplus C$-hyperarithmetic reduction, and $k$ a natural number.
			\begin{itemize}
				\item If $\X$ is a size $k$ set such that $(\Phi_0,\X)\in \Q$, and no pair $q,r\leq_\Q (\Phi_0,\X)$ form an $\ang{e,\alpha}$-split, then $\U$ has an infinite path on which $A$ does not appear.
				\item If $\U$ has an infinite path on which $A$ does not appear, then each such path $Y$ is identified with a size $k$ set $\X(Y)$ such that $(\Phi_0,\X(Y))\in\Q$ has no $\ang{e,\alpha}$-splits below it.
			\end{itemize}
			\label{ess_splits_hard_1}
		\end{lem}
		
		\begin{proof}
			The strategy is similar to that of Lemma \ref{ess_splits_easy_1}, with the same modifications needed to turn a proof of Lemma \ref{ess_sent_easy_1} into one of Lemma \ref{ess_sent_hard_1}.
		\end{proof}
		
		\begin{lem}
			If $\Phi_0$ is a Turing functional which partially computes $B$ on input $A$, $\ang{e,\alpha}$ a $(B\oplus C)$-hyperarithmetic reduction, and $k$ is a natural number, then $\U$ is $\Pi_{\alpha+3}(B\oplus C)$, uniformly in the defining parameters. Furthermore if there exists a size $k$ set $\X$, such that $(\Phi_0,\X)\in \Q$ has no $\ang{e,\alpha}$-splits below it, then we can find such an $\X$ not containing $A$ all of whose members are recursive in $(B\oplus C)^{(\alpha+4)}$ (it is also uniform, although this is not needed).
			\label{ess_splits_hard_2}
		\end{lem}
		
		\begin{proof}
			The strategy is similar to that of Lemmas \ref{ess_splits_easy_2} and \ref{ess_splits_easy_3}. Fix $\Phi_0,k,\ang{e,\alpha}$ and $\btau$. Then $\btau \in \U$ iff
		\begin{align*}
		(\forall \Phi_q,\Phi_r\leq_\Q \Phi_0)(\forall m,m_1, m_2)
		& [ (\exists \X_q,A\notin \X_q)((\Phi_q,\X_q)\forces \ang{e,\alpha}^{(B\oplus C \oplus \Phi_G)}(m)\downarrow =m_1)\\ 
		&\land (\exists \X_r,A\notin \X_r)((\Phi_r,\X_r) \forces \ang{e,\alpha}^{(B\oplus C\oplus \Phi_G)}(m) \downarrow = m_2) \\ 
		&\land m_1\neq m_2) \\
		& \rightarrow (\exists\ang{x,y,\sigma}) \text{ such that }\ang{x,y,\sigma}\in (\Phi_q\cup \Phi_r)\setminus \Phi_0\\
		&\hspace{8mm}\text{yet $\sigma$ is compatible with some }\tau\in\btau)]
		\end{align*}

	By Lemma \ref{ess_splits_hard_1} we can replace the existential quantifiers over reals by 
		\[
			(\exists j)U(\Phi_q,\ang{e,\alpha}^{(B\oplus C \oplus \Phi_G)}(m)\downarrow = m_1,j) \text{ has a path on which $A$ does not appear.}
		\]
	 (making the analogous changes for $\Phi_r$). Now, the formula "$\ang{e,\alpha}^{(B\oplus C \oplus \Phi_G)}(m)\downarrow = m_1$" can be rendered as a $\Sigma_{\alpha+1}^r(B\oplus C)$ formula in our forcing language so, increasing its level by $1$ to get a $\Pi^r(B\oplus C)$ formula, and applying Lemma \ref{ess_sent_hard_2}, the existential quantifier over reals is equivalent to a (uniformly given) $\Pi^0_{\alpha+2}(B\oplus C)$ tree having a path on which $A$ does not appear.
	
	We can decided whether the tree has such a path by answering the question: Does there exist a $j$ and a node on a uniformly given $\Pi_{\alpha+2}^0(B\oplus C)$ subtree of a recursively bounded tree, the coordinates of which all disagree with $A$, above which the tree is infinite? Which is a uniformly $\Sigma_{\alpha+3}^0(B\oplus C)$ question.
	
	This means we can express "$\btau$ is essential" as a formula with a prefix of universal quantifiers, then a uniformly given $\Pi_{\alpha+3}^0(B\oplus C)$ matrix, and so this property is $\Pi_{\alpha+3}^0(B\oplus C)$, as required.
	
	So, to find a "simple" $\X$, provided there is one, we need to be able to construct the tree $\U$ and find paths on it, on which $A$ does not appear. As $\U$ is a $\Pi^0_{\alpha+3}(B\oplus C)$ subtree of a recursively bounded recursive tree $(B\oplus C)^{(\alpha+4)}$ can decide whether a given $\btau$ is on the tree and if so whether the tree is infinite above that node. It can also decide whether the node has a coordinate which is an initial segment of $A$, and so can construct a the desired path provided one exists.
\end{proof} 

\begin{cor}
	Let $q\in \Q$ be a forcing condition and $\ang{e,\alpha}$ a $(B\oplus C)$-hyperarithmetic reduction. Then we can find an extension $r\in \Q$ of $q$ which diagonalizes against $\ang{e,\alpha}^{B\oplus C\oplus \Phi_G} = D$ i.e. $p$ either forces $\ang{e,\alpha}^{B\oplus C\oplus\Phi_G}$ is not total, or, there is an $x$ such that it forces this computation on $x$ to be convergent, yet not equal to $D(x)$.
	\label{diagonalising_hard}
\end{cor}
\begin{proof}
	The idea is the same as in Corollary \ref{diagonalising_easy}. We assume there is a condition $q$ which forces totality and has no $\ang{e,\alpha}$-splits. We then move to the related condition $q'$ which has the same finite part as $q$, but the real parts are all recursive in $(B\oplus C)^{(\alpha+4)}$ as guaranteed by Lemma \ref{ess_splits_hard_2}. We then search for extensions of $q'$ forcing $\ang{e,\alpha}^{B\oplus C\oplus \Phi_G}(x)\downarrow = y$ and output $y$ if we find it. We only need to search for conditions with new reals recursive in $(B\oplus C)^{(\alpha + 1)}$ by Corollary \ref{ess_sent_hard_4}. This makes the search hyperarithmetic in $B\oplus C$, and so, if it computes $D$, then $B\oplus C$ computes $D$, contrary to our hypothesis.
\end{proof}
\begin{cor}
	The $\Sigma_2$ theory of the arithmetic degrees in $\leq, \sqcup$ and $0$ is decidable.
\end{cor}
\begin{proof}
	By Jockusch and Slaman, it suffices to establish the extended Posner-Robinson theorem. Choose representatives for each of the arithmetic degrees $\a,\b,\c_i,\d_i$ such that if $i$ is a hard case then $A\leq_T C_i$. 
	
	One constructs a sequence (beginning with $(\emptyset,\emptyset)$) of conditions $(\Phi_i,\X_i)$ maintaining that $\Phi_i$ partially computes $B$ given $A$ and that $A\notin \X_i$. Hence, our sequence will live in both $\P$ and $\Q$. We dovetail together all the $\Pi^r_n$ formulas for $n<\omega$, and relevant arithmetic reductions against which we will diagonalize, and successively extend our sequence in the corresponding notion of forcing ($\P$ for easy cases, $\Q$ for hard) deciding the next sentence, and diagonalizing against the reductions. This ensures that forcing equals truth for all arithmetic sentences of our forcing languages. By Corollary \ref{ess_sent_easy_4} and \ref{ess_splits_easy_3} we can do this for the easy cases without adding $A$ to the infinite part or incorrect axioms applying to $A$ to the finite part, as in the easy case $A$ is not arithmetic in $C$.
	
	Between these stages we add the axiom $(n,B(n),\alpha)$ for the least $n$ so that there is no axiom applying to $A$ of this form, and $\alpha$ is the shortest initial segment of $A$ permitted to enter. As we are maintaining $A\notin \X_i$ this is possible, and we never add any incorrect axioms applying to $A$. This ensures that $A \oplus \Phi_G \geq_a B$ as required.
	
	As we have diagonalized against each of the listed arithmetic reduction and forcing equals truth for each arithmetic formula, we know for each index $e$ and natural number $n$ that $\ang{e,n}^{C_i\oplus \Phi_G}\neq D_i$ in the easy case, and $\ang{e,n}^{B\oplus C_j\oplus \Phi_G}\neq D_j$ in the hard case (this relies on the observation that our algorithms in Corollaries \ref{diagonalising_easy} and \ref{diagonalising_hard} are arithmetic in the parameters when $\alpha$ is finite). For avoiding the ideals below $\e_i$, we could do more forcing arguments. For brevity though, we appeal to the fact that there are continuum many generics, even with our coding of $B$.
\end{proof}
\section{Preserving $\omega_1^{CK}$}

	To move the previous result to the hyperarithmetic setting we need to show that we can also preserve $\omega_1^{CK}$ relativised to $C$ in the easy case, and $B\oplus C$ in the hard case. We use a method suggested by Slaman: forcing over nonstandard models of ZFC.

\begin{defn}
	Let $\M$ be an $\omega$-model of ZFC. Then $\P_\M$ is the collection of all Kumabe-Slaman conditions $p=\condition{p}\in \P$ such that every real $X\in \X_p$ is in $\M$. Because $\M$ is an $\omega$-model this definition is meaningful. Similarly $\Q_{\M}$ is the restriction of the forcing $\Q = \Q_{A,B}$ to $\M$. Extension is defined as for $\P$ and $\Q$ respectively. Note that $\P_\M\in \M$ for any such $\M$, and $\Q_{\M}\in \M$ if $A$ and $B$ are also members of $\M$.
\end{defn}

\begin{nota}
	Given a Kumabe-Slaman condition $p=\condition{p}\in \P$ we call $(\Phi_p,\X_p \cap (2^\omega)^{\M})$ the restriction of $p$ to $\M$ and denote it $p\restricted_\M$.  
\end{nota}

\subsection{Easy case}

Suppose $C\in\c$ corresponds to an easy case and $\M$ is a countable $\omega$-model of ZFC that omits $\omega_1^{C}$ and $A$ but contains $C$. The existence of such a model is shown in \cite{ShoreSlaman2016}.

\begin{lem}
	If $(p_i)_{i=0}^\infty$ is an $\M$-generic sequence (i.e. the sequence meets all of the dense subsets of $\P_\M$ that appear in $\M$), then $\Phi_G$ preserves $\omega_1^{C}$, that is to say, $\omega_1^{C\oplus \Phi_G} = \omega_1^{C}$.
	\label{preserve_omega1_easya}
\end{lem}

\begin{proof}
	Suppose $<$ is a well-order recursive in $C \oplus \Phi_G$. As $\M$ is an $\omega$-model and the ordinals of $\M[\Phi_G]$ are the same as that of $\M$, then $\M[\Phi_G]$ is an $\omega$-model containing $C \oplus \Phi_G$ and is closed under Turing reduction. Consequently, $\M[\Phi_G]$ contains an isomorphic copy of $<$, and as $<$ is externally well-founded, it is well-founded in $\M[\Phi_G]$.
	
	Hence, $\M[\Phi_G]$ thinks $<$ is isomorphic to an ordinal $\alpha$, and as $<$ is externally well-founded, so too is $\alpha$. However, the only externally well-founded ordinals in $\M[\Phi_G]$ are those that are also externally well-founded ordinals in $\M$, and by hypothesis, those are precisely the ordinals less than $\omega_1^C$. Therefore, $<$ has height less than $\omega_1^C$ and so $\omega_1^{C\oplus \Phi_G} = \omega_1^{C}$ as required. 
\end{proof}

\begin{lem}
	Suppose $\mathcal{D} \subseteq\P_{\M}$ is dense, and also an element of $\M$. Further suppose $p=\condition{p}\in \M$ is a forcing condition. Then we can find an extension $q\in \P_{\M}$ of $p$ meeting $\mathcal{D}$ such that $q$ adds no new computations applying to $A$, and $A$ is not in the infinite part of $q$ (as $A\notin \M$).
	\label{preserve_omega1_easyb}
\end{lem}

\begin{proof}
	Suppose we can't meet $\mathcal{D}$ below $p$ without adding new computations to $A$. We will show that $\mathcal{D}$ is not dense. Let $\btau$ be a finite sequence of finite binary strings all of the same length. We say $\btau$ is essential to meeting $\mathcal{D}$ over $\Phi_p$ if for every condition $r=\condition{r}\in \mathcal{D}$ with $r <(\Phi_p,\emptyset)$, there is an axiom $\ang{x,y,\sigma}\in \Phi_r\setminus \Phi_p$ with $\sigma$ compatible with some component of $\btau$.
	
	Let $V(\Phi_p,\mathcal{D},k)$ denote the set of $\btau$ of length $k$ which are essential to meeting $\mathcal{D}$ over $\Phi_p$ and note $V(\Phi_p,\mathcal{D},k)\in \M$ for all $k$. Let $k = |\X_p|$ and enumerate $\X_p = (X_1,\ldots,X_k)$. We claim $\btau_l = (X_1\restricted l, \ldots, X_k\restricted l, A\restricted l) \in V(\Phi_p,\mathcal{D},k+1)$ for each $l$. To see this, suppose $r=\condition{r}\in \P_\M$ properly extends $(\Phi_p,\emptyset)$ and meets $D$. If $r$ has no new axioms applying to any of $X_1,\ldots,X_k,A$, then $r$ extends $p$, contradicting our hypothesis that we can't meet $D$ below $p$ without adding axioms about $A$. Hence, such an $r$ has a new axiom $\ang{x,y,\sigma}$ applying to one of $X_1,\ldots,X_k,A$ and so $\sigma$ is compatible with the corresponding component of $\btau_l$, showing $\btau_l$ is essential.
	
	Consequently, $V(\Phi_p,\mathcal{D},k+1)$ is an infinite finitely-branching tree that is an element of $\M$. As $\M$ is an $\omega$-model it sees that $V(\Phi_p,\mathcal{D},k+1)$ is infinite and finitely-branching, so by K\"onig's Lemma inside $\M$, $V(\Phi_p,\mathcal{D},k+1)$ has a branch $Y$ in $\M$. This branch corresponds to reals $\X(Y)$ which are also in $\M$ and so $q = (\Phi_p,\X_p\cup \X(Y))\in\P_\M$ and extends $p$.	

	As $\mathcal{D}$ is dense there is some extension $r=\condition{r}$ of $q$ in $\mathcal{D}$. Such an $r$ extends $(\Phi_p,\emptyset)$ so by the definition of essentiality, there is some axiom $\ang{x,y,\sigma}\in \Phi_r\setminus \Phi_p$ with $\sigma$ compatible with some component of each element of $Y$. Therefore, $\sigma$ is an initial segment of some $X\in \X(Y)$, but the axiom $\ang{x,y,\sigma}$ can only be in an extension of $q = (\Phi_p,\X_p \cup \X(Y))$ if it was already in $\Phi_p$, a contradiction. Therefore, $q$ has no extensions in $\mathcal{D}$ and so $\mathcal{D}$ is not dense. 
\end{proof}

\subsection{Hard case}
	If $A\leq_h C$, then we can't omit $A$ from $\omega$-models of ZFC containing $C$. Consequently, we can't rely on the fact that $A\notin \M$ to prevent us from adding $A$ to the infinite part of a condition. Indeed, if $\M$ contains $A$, then $\mathcal{D} = \{p \in \P : A\in \X_p\}$ is a dense subset of $\P_\M$ which is also a member of $\M$. Consequently, we have to use a different forcing.
	
	Of course given what has come before, we want to use $\Q_\M$. Any $\omega$-model which contains $C$, must contain $A$ in the hard case, and consequently, $\Q_\M$ is an element of such a model iff $B$ is. Therefore, in the hard case, we pick an $\omega$-model $\M$ which contains $C$ and $B$, yet omits $\omega_1^{B\oplus C}$, again this follows from\cite{ShoreSlaman2016}.
	
	Once one has such a model, the argument for any set forcing preserving the least omitted ordinal (in this case $\omega_1^{B\oplus C}$) is the same as above. By definition of the forcing we need not worry about interfering with our coding procedure.
	
\subsection{The compatibility lemma}

	For a condition $p=\condition{p}$ in either of the restricted forcings $\P_\M$ and $\Q_\M$, it is still a fact that no extension of $p$ (in the relevant forcing) can add new axioms about members of $\X_p$. However, as the class of reals from which we can pick members of $\X_p$ from is highly restricted, we want to know we can meet dense sets without adding axioms applying to reals not appearing in the models (to ensure compatibility with the unrestricted forcings).
	
	\begin{lem}	
		Suppose $C$ is an easy case, then for every countable $\omega$-model of ZFC $\M$ omitting $\omega_1^C$ and $A$ yet containing $C$, every finite set of reals $\X$, every condition $p=\condition{p}\in \P_{\M}$, and every dense $\mathcal{D}\subseteq \P_\M$ which is also an element of $\M$, there is an extension $r\in\P_\M$ of $p$ meeting $\mathcal{D}$ which adds no new computations to any member of $\X$. 
			
		On the other hand if $C$ is a hard case, then for every countable $\omega$-model of ZFC $\M$ omitting $\omega_1^{B\oplus C}$ yet containing $B$ and $C$ (and, consequently, $A$), every finite set of reals $\X$ not containing $A$, every condition $q\in \Q_\M$, and every dense $\mathcal{D}\subseteq \Q_\M$ which is also an element of $\M$, there is an extension $r\in\Q_\M$ of $q$ meeting $\mathcal{D}$ which adds no new computations to any member of $\X$.
		\label{compataibilty_lemma}
	\end{lem}
	
	\begin{proof}
		For the first claim fix such an $\M$, an $\X$, a $p\in\P_\M$, and a $\D\in\M$. If any $X\in \X$ is also an element of $\M$, then the condition $p' = (\Phi_p, \X_p \cup (\X\cap \M))\in \P_\M$ extends $p$ and can add no new computations. By the density of $\D$ we can meet $\D$ below $p'$ adding no computations to any $X\in \X\cap \M$, so we assume that $\X$ is disjoint from $\M$.
		
		Now suppose that every extension of $p$ in $\D$ adds a new axiom applying to some $X\in \X$. As in Lemma \ref{preserve_omega1_easyb} consider the tree $V(\Phi_p,\D,|\X|)$ of essential $\btau$ of length $|\X|$. By assumption $\btau_l = (X_1\restricted l,\ldots,X_n\restricted l)$ is essential to meeting $\D$ below $p$ (where $\X = \{X_1,\ldots X_n\})$. Consequently, $V(\Phi_p,\D,|\X|)$ is a finitely branching tree in $\M$ which is infinite. Hence, it has a path $Y$ in $\M$ corresponding to reals $\X(Y)$ also in $\M$.
		
		We claim $(\Phi_p, \X(Y))$ has no extensions in $\D$. If it did, such an $r$ has a new axiom compatible with some component of each member of $Y$. Consequently, $r$ has a new axiom applying to some $X\in \X(Y)$ and so $r$ can't extend $(\Phi_p,\X(Y))$ a contradiction.
		
		The proof for the second claim in the lemma is similar: as we are assuming $\X$ does not contain $A$, we can still use the same trick as before to assume $\X$ is disjoint from $\M$. The proof from there is the same.
	\end{proof}

\section{Bringing it all together}
	Fix representatives $A\in \a$ and $B\in \b$. Then for each $\c_i,\d_i$ pair, pick representatives $C_i$ and $D_i$, with the requirement that if $i$ is a hard case, then $A\leq_T C_i$. Then, for each easy case $i$, fix a countable $\omega$-model $\M_i$ of ZFC containing $C_i$ yet omitting $\omega_1^{C_i}$ and $A$, and for each hard case $j$, fix a countable $\omega$-model $\M_j$ of ZFC containing $B$ and $C_i$ yet omitting $\omega_1^{B\oplus C_i}$.
	
	Now we need to show that we can construct a generic for all the forcings simultaneously. We can't construct a single generic sequence, because we may add a real $X$ to the infinite part at some point, where that $X$ is not a member of some $\M_i$. To work around this we use the Compatibility Lemma.
	
	We will have one master sequence $\{(\Phi_n,\X_n)\}_{n=0}^\infty$ which will live in $\P$ the unrestricted forcing. We will maintain that $A\notin \X_n$, $\Phi_n$ partially computes $B$ on input $A$ and $p_{n+1}\leq_\P p_n$. These three conditions will guarantee that our sequence is also a descending sequence of $\Q$ conditions. The master sequence will be used to decide all the sentences of the languages $\L^r_{\omega_1,\omega}(C_i)$ and $\L^r_{\omega_1,\omega}(B\oplus C_j)$ for easy cases $i$ and hard cases $j$. We will also diagonalize against computing $D_i$. This part is the same as the argument for the arithmetic degrees, and Corollaries \ref{ess_sent_easy_4},\ref{ess_sent_hard_4}, \ref{diagonalising_easy} and \ref{diagonalising_hard} say we can build such a sequence.
	
	To meet the dense sets in one of our $\omega$-models $\M_i$, we consider the sequence of restricted conditions $(\Phi_n,\X_n)\restricted \M_i$. Certainly at some stage of the construction, we can extend the condition $(\Phi_n,\X_n)\restricted \M_i$ in the corresponding forcing living in $\M_i$ to meet a dense set, however, we can do better. 
	
	By lemma \ref{compataibilty_lemma}, we can extend $(\Phi_n,\X_n)\restricted \M_i$ in $\M_i$ to meet $\D$, without adding any axioms to $A$ or any member of $\X_n$, nor adding $A$ to the infinite part. Call such an extension $q=\condition{q}$. Then we would define the next member of the master sequence to be $p_{n+1}=(\Phi_q,\X_{q}\cup \X_{p_{n}})$. By hypothesis $p_{n+1}\leq_\P p_n$, $p_{n+1}$ partially computes $B$ given $A$, and doesn't have $A$ in the infinite part. Thus, we can continue coding, and have maintained our invariants.
	
	Thus, our master sequence $(\Phi_n,\X_n)$ decides each sentence of each $\L^r_{\omega_1,\omega}(C_i)$ and $\L^r_{\omega_1,\omega}(B\oplus C_j)$, and diagonalizes against every $C_i$ or $B\oplus C_j$-hyperarithmetic reduction, depending on the case. The restricted sequence $(\Phi_n,\X_n)\restricted \M_i$ is $\M_i$ generic, and corresponds to the same generic object as for the master sequence. Hence, $\Phi_G$ preserves each $\omega_1^{C_i}$ and $\omega_1^{B\oplus C_j}$, which completes the proof of the extended Posner-Robinson theorem.

\pagebreak

\end{document}